\documentclass[11pt,reqno]{article}

\usepackage[T1]{fontenc}
\usepackage[utf8]{inputenc}
\usepackage{lmodern}
\usepackage{microtype}
\usepackage{setspace}
\usepackage{amsmath,amssymb,amsthm,amsfonts,mathtools}
\usepackage{enumitem}
\numberwithin{equation}{section}

\usepackage{graphicx}
\usepackage{float}
\usepackage{subcaption}

\usepackage[margin=1in]{geometry}
\usepackage{hyperref}
\hypersetup{
    colorlinks=true,
    linkcolor=blue,
    citecolor=blue,
    urlcolor=blue
}
\usepackage[nameinlink,noabbrev]{cleveref}

\usepackage{makeidx}
\makeindex

\usepackage{mathrsfs}

\theoremstyle{plain}
\newtheorem{theorem}{Theorem}[section]
\newtheorem{proposition}[theorem]{Proposition}
\newtheorem{corollary}[theorem]{Corollary}
\newtheorem{lemma}[theorem]{Lemma}

\theoremstyle{definition}
\newtheorem{definition}[theorem]{Definition}
\newtheorem{remark}[theorem]{Remark}

\theoremstyle{plain}
\newtheorem*{theoremx}{Theorem}
\newtheorem*{propositionx}{Proposition}
\newtheorem*{corollaryx}{Corollary}
\newtheorem*{lemmax}{Lemma}

\theoremstyle{definition}
\newtheorem*{definitionx}{Definition}
\newtheorem*{remarkx}{Remark}

\newenvironment{theorem*}{\begin{theoremx}}{\end{theoremx}}
\newenvironment{proposition*}{\begin{propositionx}}{\end{propositionx}}
\newenvironment{corollary*}{\begin{corollaryx}}{\end{corollaryx}}
\newenvironment{lemma*}{\begin{lemmax}}{\end{lemmax}}
\newenvironment{definition*}{\begin{definitionx}}{\end{definitionx}}
\newenvironment{remark*}{\begin{remarkx}}{\end{remarkx}}

\newcommand{\R}{\mathbb{R}}

\newcommand{\C}{\mathbb{C}}

\newcommand{\Id}{\mathrm{Id}}

\DeclareMathOperator{\Sym}{Sym}
\DeclareMathOperator{\Spec}{Spec}
\DeclareMathOperator{\Crit}{Crit}
\DeclareMathOperator{\Hess}{Hess}

\title{
Spectral Selection and Minimal Morse Structures
on the Poincar\'e Dodecahedral Space
}

\author{
Carlos A. Cadavid \and Juan D. V\'elez \and Sergio Lenis
}

\date{\today}

\begin{document}

\maketitle

\begin{abstract}
We study the long-time behavior of the heat equation on the spherical
Poincar\'e dodecahedral space \(M=S^3/I^*\) and introduce a spectral
selection property \(P\), asserting that for a dense open set of initial
data, the solution eventually becomes a minimal Morse function.

We first establish an obstruction principle: if the first positive
eigenspace of the Laplace--Beltrami operator contains a Morse function
that is not minimal, then property \(P\) fails. Using an explicit
representation-theoretic description of the spherical first eigenspace,
we show that the round metric on \(M\) violates property \(P\).

We then develop a perturbative spectral selection mechanism. Using
conformal variations and a finite-dimensional reduction of the
first-order splitting of the lowest eigenvalue cluster, we construct
metrics arbitrarily close to the spherical metric for which the first
eigenvalue is simple and the corresponding eigenfunction is minimal
Morse with exactly six critical points. As a consequence, these nearby
metrics satisfy property \(P\).

This establishes both the failure and the restoration of minimal Morse
selection on \(M\), and provides a concrete spectral mechanism linking
representation theory, eigenvalue splitting, and global Morse structure.
\end{abstract}

\tableofcontents
\bigskip

\section{Introduction}

Let \((N,g)\) be a smooth closed Riemannian manifold, and consider the heat
equation
\[
\partial_t u=-\Delta_g u.
\]
Its large-time behavior is governed by the lowest part of the Laplace
spectrum: after subtraction of the constant mode and exponential rescaling, the
solution converges to the projection of the initial datum onto the first
nonzero eigenspace. This suggests a natural question: for generic initial data,
does the heat flow eventually select functions of minimal Morse-theoretic
complexity?

In this paper we study that question on the Poincar\'e dodecahedral space
\[
M=S^3/I^*,
\]
where \(I^*\subset SU(2)\) is the binary icosahedral group. We say that a
smooth closed Riemannian manifold \((N,g)\) has property \(P\) if there exists
an open dense subset
\[
\mathcal S\subset L^2(N,g)
\]
such that for every \(f\in\mathcal S\), the corresponding heat evolution
\(u(\cdot,t)\) is, for all sufficiently large \(t\), a minimal Morse function.

On \(M\), minimality has a concrete numerical form. Since \(M\) has Heegaard
genus \(2\), a Morse function on \(M\) is minimal if and only if it has exactly
six critical points. The problem is therefore spectral: does the first nonzero
eigenspace generically force six critical points, or can it support more
complicated Morse profiles?

For the spherical metric the answer is negative. Let
\[
g_{\mathrm{sph}}
\]
denote the round metric on \(M\). The first positive eigenvalue of
\(-\Delta_{g_{\mathrm{sph}}}\) is
\[
168,
\]
and the corresponding eigenspace
\[
E_{168}(g_{\mathrm{sph}})
\]
has dimension \(13\). Representation-theoretically, this eigenspace is the
degree-\(12\) binary-icosahedral block, equivalently the space arising from the
unique \(I^*\)-invariant line in \(\Sym^{12}(\C^2)\). This structure makes the
round case explicit enough to analyze in detail, but also highly degenerate
from the spectral point of view.

The first part of the paper proves that this degeneracy obstructs property
\(P\). We first establish an abstract principle valid on any closed manifold:
if the first nonzero eigenspace contains a Morse function that is not minimal,
then property \(P\) fails. We then apply this to the spherical metric on
\(M\). Starting from the distinguished \(I^*\)-invariant first eigenfunction
\(F_0\), we determine its full critical set and show that it consists exactly
of the three exceptional Hopf fibers of orders \(5\), \(3\), and \(2\), with
Morse--Bott nondegeneracy along each one. We then construct an explicit
splitting direction along the order-\(2\) fiber and use it to produce a Morse
function in \(E_{168}(g_{\mathrm{sph}})\) with twelve critical points. It
follows that
\[
(M,g_{\mathrm{sph}})
\]
does not satisfy property \(P\).

The second part shows that this failure is unstable under perturbation. We
consider conformal deformations
\[
g_\varepsilon=e^{2\varepsilon\rho}g_0,
\qquad g_0=g_{\mathrm{sph}},
\]
and study the first-order splitting of the spherical first eigenspace
\[
E=E_{168}(g_0).
\]
The splitting is encoded by a finite-dimensional space
\[
\mathscr B\subset \Sym(E)
\]
of realizable symmetric operators obtained by compressing multiplication by
\[
q_\rho=2\lambda_{\mathrm{sph}}\rho+\frac12\Delta_{g_0}\rho.
\]
Inside \(\mathscr B\) we construct a canonical seed operator \(A_0\) whose
simple lowest eigendirection is the invariant line \(\mathbb RF_0\), and we
prove that the lowest-eigenline map is a submersion at \(A_0\). Consequently,
every projective line sufficiently close to \([F_0]\) is realized as the simple
lowest eigendirection of some first-order conformal splitter.

The remaining task is geometric: one must construct minimal-Morse lines in
\(E\) arbitrarily close to \([F_0]\). This is achieved by exploiting the global
critical-circle structure of the invariant seed. We introduce three explicit
coefficient functions
\[
F_1,\qquad F_3,\qquad F_4
\]
in the degree-\(12\) matrix-coefficient model. Their weights force each one to
survive on exactly one exceptional fiber, and the surviving restriction is a
nonzero quotient mode of frequency \(1\). Small perturbations by
\[
\Re(F_1),\qquad \Re(F_3),\qquad \Re(F_4)
\]
therefore split the three Morse--Bott critical circles independently, producing
exactly two nondegenerate critical points on each circle and no others. In this
way we obtain explicit nearby lines in \(E\) generated by Morse functions with
exactly six critical points.

Combining this geometric construction with local realizability and elliptic
perturbation theory for a simple first branch, we prove that there exist smooth
metrics arbitrarily \(C^\infty\)-close to \(g_{\mathrm{sph}}\) such that:
\begin{enumerate}[label=\arabic*.]
\item the first positive eigenvalue is simple;
\item the corresponding first eigenfunction is Morse with exactly six critical
points;
\item hence that first eigenfunction is minimal Morse;
\item consequently the metric has property \(P\).
\end{enumerate}

Thus the spherical metric on the Poincar\'e dodecahedral space fails property
\(P\) because maximal symmetry permits non-minimal Morse configurations inside a
large first eigenspace. But this failure disappears under arbitrarily small
perturbations: nearby metrics recover a simple first eigenvalue and a
minimal-Morse first eigenfunction, so the heat flow generically selects minimal
topology in the large-time regime.

\medskip

\noindent
The paper is organized as follows. Part~I proves the abstract obstruction
theorem and applies it to the spherical metric by constructing a non-minimal
Morse function in the first eigenspace. Part~II develops the perturbative
theory: realizable first-order conformal splitting, local realizability of
lowest eigenlines, explicit construction of nearby minimal-Morse directions, and
the perturbative existence theorem beyond the spherical metric.

\part{The spherical metric fails property \(P\)}
\label{part:spherical-obstruction}

Throughout this part, let
\[
M=S^3/I^*
\]
be the Poincar\'e dodecahedral space, and let \(g_0=g_{\mathrm{sph}}\) denote
its spherical metric.

\section{Property \(P\) and an abstract obstruction theorem}
\label{sec:partI-property-P}

\begin{definition}
Let \((N,g)\) be a smooth closed connected Riemannian manifold. We say that
\((N,g)\) has property \(P\) if there exists an open dense subset
\[
\mathcal S\subset L^2(N,g)
\]
such that for every \(f\in\mathcal S\), the heat evolution \(u(\cdot,t)\) with
initial condition \(u(\cdot,0)=f\) is, for all sufficiently large times, a
minimal Morse function.
\end{definition}

The first step is an abstract obstruction principle.

\begin{theorem}[First-eigenspace obstruction]
\label{thm:partI-obstruction}
Let \((N,g)\) be a smooth closed connected Riemannian manifold. Assume that the
first nonzero eigenspace \(E_{\lambda_1}(g)\) contains a Morse function which is
not minimal. Then \((N,g)\) does not have property \(P\).
\end{theorem}

\begin{proof}
Let
\[
\Pi_1:L^2(N,g)\to E_{\lambda_1}(g)
\]
be the orthogonal projection onto the first eigenspace, and let
\[
u\in E_{\lambda_1}(g)
\]
be Morse and non-minimal.

Since the Morse property is open in the \(C^2\)-topology, and the number of
critical points of a Morse function is locally constant in the \(C^2\)-topology,
there exists a \(C^2\)-neighborhood \(\mathcal U\subset C^\infty(N)\) of \(u\)
such that every \(v\in\mathcal U\) is Morse and has the same number of critical
points as \(u\). In particular, every \(v\in\mathcal U\) is again non-minimal.

Define
\[
\mathcal V=\Pi_1^{-1}\bigl(\mathcal U\cap E_{\lambda_1}(g)\bigr)\subset L^2(N,g).
\]
Since \(\Pi_1\) is continuous and \(\mathcal U\cap E_{\lambda_1}(g)\) is a
nonempty relatively open subset of the finite-dimensional space
\(E_{\lambda_1}(g)\), the set \(\mathcal V\) is a nonempty open subset of
\(L^2(N,g)\).

Fix \(f\in\mathcal V\), and let \(w(\cdot,t)\) be the heat evolution with
initial condition \(f\). Write the spectral decomposition of \(f\) as
\[
f=f_0+f_1+f_{\ge2},
\]
where
\[
f_0\in E_{\lambda_0}(g),\qquad
f_1=\Pi_1f\in E_{\lambda_1}(g),\qquad
f_{\ge2}\in \overline{\bigoplus_{k\ge2}E_{\lambda_k}(g)}.
\]
Then
\[
w(\cdot,t)=f_0+e^{-\lambda_1t}f_1+\sum_{k\ge2}e^{-\lambda_kt}f_k.
\]
Hence
\[
e^{\lambda_1t}\bigl(w(\cdot,t)-f_0\bigr)
=
f_1+\sum_{k\ge2}e^{-(\lambda_k-\lambda_1)t}f_k
\longrightarrow f_1
\]
in \(C^\infty(N)\), and therefore in \(C^2(N)\), as \(t\to\infty\).

Since \(f_1\in\mathcal U\cap E_{\lambda_1}(g)\), it follows that for all
sufficiently large \(t\),
\[
e^{\lambda_1t}\bigl(w(\cdot,t)-f_0\bigr)\in \mathcal U.
\]
Subtracting the constant \(f_0\) and multiplying by the positive scalar
\(e^{-\lambda_1t}\) do not change the critical set or Morse character.
Therefore \(w(\cdot,t)\) itself is Morse and non-minimal for all sufficiently
large \(t\).

Thus every \(f\in\mathcal V\) fails the defining property of \(P\). Since
\(\mathcal V\) is a nonempty open subset of \(L^2(N,g)\), property \(P\) fails.
\end{proof}

\section{Minimal Morse functions on the Poincar\'e dodecahedral space}
\label{sec:partI-minimal-morse}

We now specialize to \(M=S^3/I^*\).

\begin{proposition}
\label{prop:partI-minimal-six}
A Morse function on \(M\) is minimal if and only if it has exactly six critical
points.
\end{proposition}

\begin{proof}
The Poincar\'e dodecahedral space has Heegaard genus \(2\); see, for example,
\cite{Schultens,Boileau}. For a closed connected orientable \(3\)-manifold,
Morse functions with exactly one minimum and one maximum and a total of
\(2h+2\) critical points correspond to Heegaard splittings of genus \(h\); see,
for example, \cite[Chap.~6]{Schultens} or \cite[Chap.~3]{Matsumoto}. Hence the
least possible number of critical points of a Morse function on \(M\) is
\[
2\cdot 2+2=6.
\]
Therefore a Morse function on \(M\) is minimal if and only if it has exactly six
critical points.
\end{proof}

Thus, in order to apply \Cref{thm:partI-obstruction} to the spherical metric, it
is enough to construct a Morse function in the first eigenspace having at least
eight critical points.

\section{The first eigenspace of the spherical metric}
\label{sec:partI-first-eigenspace}

\begin{proposition}
\label{prop:partI-first-eigenvalue}
For the spherical metric \(g_0=g_{\mathrm{sph}}\) on \(M=S^3/I^*\), the first
nonzero eigenvalue is
\[
\lambda_1(g_0)=168,
\]
and the corresponding eigenspace
\[
E=E_{168}(g_0)
\]
has dimension \(13\).
\end{proposition}

\begin{proof}
This is the classical binary-icosahedral invariant-theoretic description of the
spherical spectrum: the first nonconstant \(I^*\)-invariant harmonic modes on
\(S^3\) occur in degree \(12\), hence on the quotient they have eigenvalue
\(12(12+2)=168\), and the corresponding eigenspace has dimension \(13\); see
\cite{WeeksEigenmodes2006,NashHarmonics}.
\end{proof}

Inside \(E\) there is a distinguished \(I^*\)-invariant line. Let
\[
F_0\in E
\]
be a nonzero generator of this line. Geometrically, \(F_0\) is the Hopf lift of
the unique \(A_5\)-invariant spherical harmonic of degree \(6\) on \(S^2\).

\section{The invariant sextic and the order-\(2\) exceptional fiber}
\label{sec:partI-invariant-sextic}

Since the \(A_5\)-invariant degree-\(6\) spherical harmonic is unique up to
scale, we may work with any nonzero explicit \(A_5\)-invariant sextic and then
harmonically correct it. Set
\[
P(x,y,z)
=
(\tau^2x^2-y^2)(\tau^2y^2-z^2)(\tau^2z^2-x^2),
\qquad
\tau=\frac{1+\sqrt5}{2}.
\]

\begin{lemma}
\label{lem:partI-harmonic-correction}
There exists a constant \(c_\tau\in\R\) such that
\[
\widetilde P(x,y,z)=P(x,y,z)+c_\tau(x^2+y^2+z^2)^3
\]
is harmonic on \(\R^3\). In particular, the restriction \(P|_{S^2}\) differs by
a constant from the unique \(A_5\)-invariant spherical harmonic of degree \(6\).
\end{lemma}

\begin{proof}
Because \(P\) is homogeneous of degree \(6\), its Euclidean Laplacian is
homogeneous of degree \(4\). Since \(P\) is \(A_5\)-invariant,
\[
\Delta_{\R^3}P
\]
is an \(A_5\)-invariant quartic. The space of \(A_5\)-invariant quartics is
one-dimensional, generated by \((x^2+y^2+z^2)^2\). Therefore
\[
\Delta_{\R^3}P=\kappa_\tau(x^2+y^2+z^2)^2
\]
for some \(\kappa_\tau\in\R\).

On the other hand,
\[
\Delta_{\R^3}(x^2+y^2+z^2)^3=42(x^2+y^2+z^2)^2.
\]
Thus choosing
\[
c_\tau=-\frac{\kappa_\tau}{42}
\]
makes \(\widetilde P\) harmonic. Since \((x^2+y^2+z^2)^3=1\) on \(S^2\), the
restrictions of \(P\) and \(\widetilde P\) to \(S^2\) differ only by a
constant, and hence define the same spherical harmonic line.
\end{proof}

Let
\[
p_2
=
\left(\frac{1}{2\tau},\frac12,\frac{\tau}{2}\right)\in S^2.
\]
This point is an edge midpoint of the icosahedral orbit structure, hence
represents an order-\(2\) point on the orbifold quotient \(S^2/A_5\).

\begin{lemma}
\label{lem:partI-p2-critical}
The point \(p_2\) is a critical point of \(P|_{S^2}\).
\end{lemma}

\begin{proof}
Write
\[
A=\tau^2x^2-y^2,\qquad
B=\tau^2y^2-z^2,\qquad
C=\tau^2z^2-x^2,
\]
so that \(P=ABC\).

At
\[
p_2=\left(\frac{1}{2\tau},\frac12,\frac{\tau}{2}\right),
\]
one has
\[
x^2=\frac{1}{4\tau^2},\qquad
y^2=\frac14,\qquad
z^2=\frac{\tau^2}{4},
\]
hence
\[
A(p_2)=\tau^2\frac{1}{4\tau^2}-\frac14=0,
\qquad
B(p_2)=\tau^2\frac14-\frac{\tau^2}{4}=0.
\]
Therefore
\[
\nabla P
=
(\nabla A)BC+A(\nabla B)C+AB(\nabla C)
\]
vanishes at \(p_2\). Thus \(p_2\) is critical for \(P|_{S^2}\).
\end{proof}

\begin{lemma}
\label{lem:partI-p2-hessian}
The constrained Hessian of \(P|_{S^2}\) at \(p_2\) is nondegenerate and
indefinite.
\end{lemma}

\begin{proof}
Since \(\nabla P(p_2)=0\), the spherical Hessian is the restriction of the
ambient Hessian \(\Hess P(p_2)\) to the tangent plane \(T_{p_2}S^2\).

Again write \(P=ABC\). At \(p_2\) we have
\[
A(p_2)=0,\qquad B(p_2)=0.
\]
A direct second differentiation shows that the only surviving term is
\[
\Hess P(p_2)
=
C(p_2)\bigl(\nabla A\otimes \nabla B+\nabla B\otimes \nabla A\bigr).
\]
Now
\[
\nabla A=(2\tau^2x,-2y,0),
\qquad
\nabla B=(0,2\tau^2y,-2z).
\]
Evaluating at \(p_2\) gives
\[
\nabla A(p_2)=(\tau,-1,0),
\qquad
\nabla B(p_2)=(0,\tau^2,-\tau),
\qquad
C(p_2)=\tau.
\]
Hence for \(v\in T_{p_2}S^2\),
\[
v^T\Hess P(p_2)v
=
2\tau\,(\nabla A(p_2)\cdot v)(\nabla B(p_2)\cdot v).
\]

Thus the restricted quadratic form is the product of two linear forms on the
two-dimensional tangent plane. It remains to show that these two linear forms
are not proportional on \(T_{p_2}S^2\). Suppose, to the contrary, that for some
\(\lambda,\mu\in\R\),
\[
\nabla A(p_2)-\lambda \nabla B(p_2)=\mu p_2.
\]
In coordinates this becomes
\[
(\tau,-1,0)-\lambda(0,\tau^2,-\tau)
=
\mu\left(\frac{1}{2\tau},\frac12,\frac{\tau}{2}\right).
\]
The first and third coordinates imply
\[
\mu=2\tau^2,
\qquad
\lambda=\tau^2.
\]
Substituting into the second coordinate yields
\[
-1-\tau^4=\tau^2,
\]
which is false. Thus the two restricted linear forms are not proportional.

Therefore the quadratic form
\[
v\longmapsto 2\tau\,(\nabla A(p_2)\cdot v)(\nabla B(p_2)\cdot v)
\]
has rank \(2\) and signature \((1,1)\) on \(T_{p_2}S^2\). Hence the constrained
Hessian is nondegenerate and indefinite.
\end{proof}

Let \(C_2\subset M\) denote an order-\(2\) exceptional Hopf fiber.

\begin{corollary}
\label{cor:partI-C2-morse-bott}
The function \(F_0\) is Morse--Bott along \(C_2\), and its transverse Hessian
along \(C_2\) is nondegenerate of signature \((1,1)\).
\end{corollary}

\begin{proof}
By \Cref{lem:partI-harmonic-correction}, the restriction \(P|_{S^2}\)
represents the unique \(A_5\)-invariant degree-\(6\) spherical harmonic up to
addition of a constant. Hence its Hopf lift spans the same line as \(F_0\).

Along a Hopf fiber the lifted function is constant in the fiber direction, so
the fiber tangent lies in the Hessian kernel. The transverse Hessian is exactly
the Hessian of the descended function at the corresponding orbifold point.
By \Cref{lem:partI-p2-hessian}, that transverse Hessian is nondegenerate and
indefinite. Therefore \(F_0\) is Morse--Bott along \(C_2\).
\end{proof}

\section{The full critical-circle structure of the invariant seed}
\label{sec:partI-full-critical-circle-structure}

We now complete the analysis of the distinguished invariant eigenfunction
\[
F_0\in E_{168}(g_0),
\]
showing that its critical set consists exactly of the three exceptional Hopf
fibers and that it is Morse--Bott along each of them.

Recall the explicit \(A_5\)-invariant sextic
\[
P(x,y,z)
=
(\tau^2x^2-y^2)(\tau^2y^2-z^2)(\tau^2z^2-x^2),
\qquad
\tau=\frac{1+\sqrt5}{2},
\]
whose restriction to \(S^2\) differs by a constant from the unique
\(A_5\)-invariant spherical harmonic of degree \(6\); see
\Cref{lem:partI-harmonic-correction}. Therefore \(P|_{S^2}\) has the same
critical set and the same constrained Hessians as the descended seed function on
\(S^2/A_5\), and its Hopf lift has the same critical-circle structure as \(F_0\).

\begin{lemma}
\label{lem:partI-critical-set-F0}
The critical set of \(F_0\) is exactly
\[
\Crit(F_0)=C_5\sqcup C_3\sqcup C_2,
\]
where \(C_5,C_3,C_2\) are the order-\(5\), order-\(3\), and order-\(2\)
exceptional Hopf fibers in \(M=S^3/I^*\).
\end{lemma}

\begin{proof}
It is enough to determine the critical set of \(P|_{S^2}\).

Set
\[
a=x^2,\qquad b=y^2,\qquad c=z^2,
\qquad a,b,c\ge 0,\qquad a+b+c=1,
\]
and define
\[
p(a,b,c)
=
(\tau^2 a-b)(\tau^2 b-c)(\tau^2 c-a),
\]
so that
\[
P(x,y,z)=p(x^2,y^2,z^2).
\]
Then
\[
\partial_xP=2x\,p_a,\qquad
\partial_yP=2y\,p_b,\qquad
\partial_zP=2z\,p_c.
\]

The constrained critical-point equations for \(P|_{S^2}\) are
\[
\nabla P=2\lambda(x,y,z),
\]
equivalently
\[
x(p_a-\lambda)=0,\qquad
y(p_b-\lambda)=0,\qquad
z(p_c-\lambda)=0.
\]
We classify the solutions according to the vanishing pattern of \(x,y,z\).

\medskip

\noindent\textbf{Case 1: \(xyz\neq0\).}
Then
\[
p_a=p_b=p_c,\qquad a+b+c=1.
\]
Eliminating \(a\) and \(b\) yields the quartic equation
\[
c^4-\frac43c^3+\frac7{12}c^2-\frac{19}{192}c+\frac1{192}=0.
\]
This factors as
\[
\frac1{192}(3c-1)(4c-1)(16c^2-12c+1)=0.
\]
Hence
\[
c\in\left\{\frac13,\frac14,\frac{3-\sqrt5}{8},\frac{3+\sqrt5}{8}\right\}.
\]
Solving back together with the symmetry of the equations shows that the only
solutions in this case are
\[
\left(\frac13,\frac13,\frac13\right),
\]
and the three cyclic permutations of
\[
\left(\frac1{4\tau^2},\frac14,\frac{\tau^2}{4}\right).
\]
Thus Case~1 contributes the \(8\) points
\[
\frac1{\sqrt3}(\pm1,\pm1,\pm1),
\]
and the \(24\) points obtained from sign choices and cyclic coordinate
permutations of
\[
\left(\frac1{2\tau},\frac12,\frac{\tau}{2}\right).
\]

\medskip

\noindent\textbf{Case 2: exactly one coordinate is zero.}
By symmetry, take \(z=0\) and \(x,y\neq0\). Then
\[
a+b=1,\qquad p_a=p_b.
\]
A direct substitution of \(c=0\), \(b=1-a\) into \(p_a-p_b=0\) reduces the
equation to
\[
(30+12\sqrt5)a^2-(26+10\sqrt5)a+(3+\sqrt5)=0,
\]
whose two roots are
\[
a=\frac12+\frac{\sqrt5}{10},
\qquad
a=\frac12-\frac{\sqrt5}{6}.
\]
Equivalently, the two solution pairs \((a,b)\) are
\[
\left(\frac{\tau^2}{2+\tau},\frac1{2+\tau}\right)
\]
and
\[
\left(\frac1{3\tau^2},\frac{\tau^2}{3}\right).
\]
Restoring signs and coordinate permutations gives:

\begin{itemize}
\item the \(12\) points
\[
\frac1{\sqrt{1+\tau^2}}(\pm\tau,\pm1,0)
\]
and permutations;

\item the \(12\) points
\[
\frac1{\sqrt3}(0,\pm\tau^{-1},\pm\tau)
\]
and permutations.
\end{itemize}

\medskip

\noindent\textbf{Case 3: exactly two coordinates are zero.}
Then the sphere equation forces one coordinate to be \(\pm1\). Thus we obtain
exactly the \(6\) axis points
\[
(\pm1,0,0),\qquad (0,\pm1,0),\qquad (0,0,\pm1).
\]

\medskip

Collecting the three cases, the critical set of \(P|_{S^2}\) consists exactly of:

\begin{enumerate}[label=\arabic*.]
\item the \(12\) points
\[
\frac1{\sqrt{1+\tau^2}}(\pm\tau,\pm1,0)
\]
and permutations;

\item the \(20\) points given by the \(8\) points
\[
\frac1{\sqrt3}(\pm1,\pm1,\pm1),
\]
together with the \(12\) points
\[
\frac1{\sqrt3}(0,\pm\tau^{-1},\pm\tau)
\]
and permutations;

\item the \(30\) points given by the \(6\) axis points and the \(24\) points
obtained from
\[
\left(\frac1{2\tau},\frac12,\frac{\tau}{2}\right)
\]
by sign changes and cyclic coordinate permutations.
\end{enumerate}

These are precisely the three \(A_5\)-orbits corresponding to the vertices,
face centers, and edge midpoints of the icosahedron. Therefore the descended
function on \(S^2/A_5\) has exactly three critical points, namely the orbifold
points of orders \(5,3,2\). After Hopf lift to \(M=S^3/I^*\), these become
exactly the three critical circles
\[
C_5,\qquad C_3,\qquad C_2.
\]
Hence
\[
\Crit(F_0)=C_5\sqcup C_3\sqcup C_2.
\]
\end{proof}

\begin{lemma}
\label{lem:partI-F0-morse-bott-all}
The function \(F_0\) is Morse--Bott along each of the circles
\[
C_5,\qquad C_3,\qquad C_2,
\]
and its transverse Hessian is nondegenerate along each one.
\end{lemma}

\begin{proof}
Again it is enough to work on \(S^2\) with \(P|_{S^2}\). Let \(p\in S^2\) be a
critical point of \(P|_{S^2}\), so that
\[
\nabla P(p)=2\lambda p.
\]
Then the Hessian of the restricted function \(P|_{S^2}\) at \(p\) is
\[
\bigl(\Hess_{\R^3}P(p)-2\lambda I\bigr)\big|_{T_pS^2}.
\]
We evaluate this at one representative of each critical orbit.

\medskip

\noindent\textbf{The order-\(5\) representative.}
Take
\[
p_5=\frac1{\sqrt{1+\tau^2}}(\tau,1,0).
\]
A direct computation gives
\[
\nabla P(p_5)=2\lambda_5 p_5,\qquad
\lambda_5=-\frac{6+3\sqrt5}{5}.
\]
Using the tangent basis
\[
u_5=(1,-\tau,0),\qquad v_5=(0,0,1),
\]
the restricted Hessian matrix is
\[
\begin{pmatrix}
24+\frac{56\sqrt5}{5} & 0\\[1ex]
0 & \frac{32+16\sqrt5}{5}
\end{pmatrix}.
\]
Both eigenvalues are positive. Hence the constrained Hessian is nondegenerate
and positive definite at \(p_5\).

\medskip

\noindent\textbf{The order-\(3\) representative.}
Take
\[
p_3=\frac1{\sqrt3}(1,1,1).
\]
A direct computation gives
\[
\nabla P(p_3)=2\lambda_3 p_3,\qquad
\lambda_3=\frac{2+\sqrt5}{9}.
\]
Using the tangent basis
\[
u_3=(1,-1,0),\qquad v_3=(1,1,-2),
\]
the restricted Hessian matrix is diagonal:
\[
\begin{pmatrix}
-\frac{64+32\sqrt5}{9} & 0\\[1ex]
0 & -\frac{64+32\sqrt5}{3}
\end{pmatrix}.
\]
Hence the constrained Hessian is nondegenerate and negative definite at
\(p_3\).

\medskip

\noindent\textbf{The order-\(2\) representative.}
Take
\[
p_2=\left(\frac1{2\tau},\frac12,\frac{\tau}{2}\right).
\]
Here
\[
\nabla P(p_2)=0,
\]
so the restricted Hessian is simply \(\Hess P(p_2)\) on the tangent plane.
Using the tangent basis
\[
u_2=\left(1,-\frac1\tau,0\right),\qquad
v_2=\left(1,0,-\frac1{\tau^2}\right),
\]
one obtains the matrix
\[
\begin{pmatrix}
-3\sqrt5-5 & -2\\
-2 & 1+\sqrt5
\end{pmatrix},
\]
whose determinant is
\[
(-3\sqrt5-5)(1+\sqrt5)-4=-24-8\sqrt5<0.
\]
Hence the constrained Hessian is nondegenerate and indefinite at \(p_2\).

\medskip

Thus the descended seed function on \(S^2/A_5\) has three nondegenerate
critical points, one at each orbifold point of orders \(5,3,2\). The Hopf lift
\(F_0\) is constant along the fiber direction, so at each corresponding
critical circle \(C_m\):

\begin{itemize}
\item the fiber tangent direction lies in the Hessian kernel;
\item the transverse Hessian is exactly the Hessian just computed on the base.
\end{itemize}

Therefore the Hessian kernel is exactly one-dimensional, generated by the fiber
direction, and the transverse Hessian is nondegenerate. Hence \(F_0\) is
Morse--Bott along each of \(C_5,C_3,C_2\).
\end{proof}

\begin{corollary}
\label{cor:partI-F0-morse-bott-all}
The invariant seed \(F_0\) has exactly three critical circles,
\[
\Crit(F_0)=C_5\sqcup C_3\sqcup C_2,
\]
and is Morse--Bott along each of them.
\end{corollary}

\begin{proof}
Combine \Cref{lem:partI-critical-set-F0} and
\Cref{lem:partI-F0-morse-bott-all}.
\end{proof}

\section{A genuine first-eigenspace splitting direction along \(C_2\)}
\label{sec:partI-splitting-direction}

We now replace the informal use of the raw polynomial \(I_{12}(\alpha,\beta)\)
by the corresponding coefficient function in the degree-\(12\) matrix-coefficient
model, which is a genuine element of the quotient eigenspace.

Let
\[
V_{12}=\Sym^{12}(\C^2),
\qquad
e_j(x,y)=x^{12-j}y^j,\qquad j=0,\dots,12,
\]
and let
\[
I_{12}(x,y)=x^{11}y+11x^6y^6-xy^{11}\in V_{12}.
\]
For \(z=(\alpha,\beta)\in S^3\cong SU(2)\), define
\[
I_{12,z}(x,y)=
I_{12}(\alpha x-\overline\beta y,\ \beta x+\overline\alpha y).
\]
Expanding in the monomial basis,
\[
I_{12,z}(x,y)=\sum_{j=0}^{12}A_j(z)\,e_j(x,y).
\]

\begin{proposition}
\label{prop:partI-coefficient-functions}
For each \(j=0,\dots,12\), the coefficient function \(A_j\) is right
\(I^*\)-invariant and therefore descends to a smooth complex-valued function on
\[
M=S^3/I^*.
\]
Moreover \(A_j\) belongs to the complexified first eigenspace
\[
E_\C=E\otimes_\R\C,
\]
and is a pure left-weight vector of weight \(12-2j\).
\end{proposition}

\begin{proof}
Let \(\eta=I_{12}\in V_{12}\). Since \(I_{12}\) spans the unique
\(I^*\)-invariant line in \(V_{12}\), we have
\[
\pi_{12}(h)\eta=\eta
\qquad\forall h\in I^*.
\]
Hence for \(z\in SU(2)\) and \(h\in I^*\),
\[
I_{12,zh}(x,y)
=
I_{12}\bigl(\pi_{12}(zh)(x,y)\bigr)
=
I_{12}\bigl(\pi_{12}(z)\pi_{12}(h)(x,y)\bigr)
=
I_{12}\bigl(\pi_{12}(z)(x,y)\bigr)
=
I_{12,z}(x,y).
\]
Therefore each coefficient \(A_j\) is right \(I^*\)-invariant and descends to
\(M=SU(2)/I^*\).

By Peter--Weyl, the right \(I^*\)-invariant degree-\(12\) block on \(SU(2)\)
is exactly the complexified first eigenspace \(E_\C\). Since each \(A_j\) is a
degree-\(12\) matrix coefficient with fixed right vector \(\eta\), we have
\[
A_j\in E_\C.
\]

For the weight statement, let
\[
a_t=\begin{pmatrix}e^{it}&0\\0&e^{-it}\end{pmatrix}\in SU(2).
\]
Then
\[
a_t\cdot e_j=e^{i(12-2j)t}e_j.
\]
Since the map \(v\mapsto \langle v,\pi_{12}(\,\cdot\,)\eta\rangle\) is left
\(SU(2)\)-equivariant, the descended function corresponding to \(e_j\) has left
weight \(12-2j\). Hence \(A_j\) is a pure left-weight vector of weight
\(12-2j\).
\end{proof}

\begin{lemma}
\label{lem:partI-A0-equals-raw-polynomial}
The coefficient function \(A_0\) satisfies
\[
A_0(\alpha,\beta)=I_{12}(\alpha,\beta)
\qquad\forall (\alpha,\beta)\in S^3.
\]
In particular,
\[
H_2=\Re(A_0)\in E
\]
is a well-defined real first eigenfunction on \(M\).
\end{lemma}

\begin{proof}
The coefficient \(A_0(z)\) is, by definition, the coefficient of \(x^{12}\) in
\[
I_{12,z}(x,y)=I_{12}(\alpha x-\overline\beta y,\ \beta x+\overline\alpha y).
\]
Setting \(y=0\) gives
\[
I_{12,z}(x,0)=I_{12}(\alpha x,\beta x)=x^{12}I_{12}(\alpha,\beta).
\]
Therefore the coefficient of \(x^{12}\) is exactly \(I_{12}(\alpha,\beta)\),
that is,
\[
A_0(\alpha,\beta)=I_{12}(\alpha,\beta).
\]
Since \(A_0\in E_\C\) by \Cref{prop:partI-coefficient-functions}, its real part
belongs to the real eigenspace \(E\).
\end{proof}

We next record the quotient-period calculation for exceptional Hopf fibers.

\begin{lemma}
\label{lem:partI-exceptional-fiber-period}
Let \(C_m\subset M=S^3/I^*\) be an exceptional Hopf fiber of order
\(m\in\{2,3,5\}\), and let \(z\in S^3\) lie on a lift of \(C_m\).
Then the full preimage in \(I^*\subset SU(2)\) of the stabilizer of the image
of \(z\) in \(S^2=\mathbb CP^1\) is cyclic of order \(2m\). Equivalently, on
the lifted Hopf circle
\[
t\longmapsto e^{it}z,
\]
the quotient circle \(C_m\) is obtained by identifying
\[
t\sim t+\frac{2\pi}{2m}.
\]
Consequently, if \(G\in E_\C\) is a pure left-weight vector of weight \(\ell\),
then along the lifted Hopf circle one has
\[
G(e^{it}z)=e^{i\ell t}G(z),
\]
and the restriction of \(G\) descends nontrivially to \(C_m\) only if
\(2m\mid \ell\). In that case the descended quotient frequency is
\[
\frac{\ell}{2m}.
\]
\end{lemma}

\begin{proof}
The Hopf map \(S^3\to S^2=\mathbb CP^1\) is \(SU(2)\)-equivariant with kernel
\(\{\pm I\}\). An orbifold point of order \(m\) in \(S^2/A_5\) has stabilizer
of order \(m\) in \(A_5\), so its full preimage in the binary icosahedral group
\(I^*\subset SU(2)\) has order \(2m\). Since this stabilizer fixes the complex
line \(\C z\subset \C^2\), it acts on the lifted Hopf fiber through \(z\) by
multiplication by \(2m\)-th roots of unity. Hence the quotient circle is
obtained by identifying the Hopf parameter modulo \(2\pi/(2m)\).

If \(G\) has left weight \(\ell\), then by definition
\[
G(e^{it}z)=e^{i\ell t}G(z).
\]
This descends to the quotient circle if and only if it is invariant under the
identification \(t\mapsto t+2\pi/(2m)\), namely if and only if
\[
e^{i\ell 2\pi/(2m)}=1,
\]
equivalently \(2m\mid \ell\). When this holds, the quotient frequency is
\(\ell/(2m)\).
\end{proof}

\begin{lemma}
\label{lem:partI-H2-restriction}
The restriction of \(H_2=\Re(A_0)\) to an order-\(2\) exceptional fiber \(C_2\)
descends to a nonconstant real trigonometric mode of quotient frequency \(3\).
Consequently \(H_2|_{C_2}\) has exactly six nondegenerate critical points on the
quotient circle.
\end{lemma}

\begin{proof}
Choose the lift
\[
z_2=\frac{1}{\sqrt2}(i,1)\in S^3,
\]
whose Hopf image is the edge-midpoint point \((0,1,0)\in S^2\), hence lies on
an order-\(2\) exceptional fiber.

By \Cref{lem:partI-A0-equals-raw-polynomial},
\[
A_0(z_2)=I_{12}\!\left(\frac{i}{\sqrt2},\frac{1}{\sqrt2}\right).
\]
A direct computation, recorded in
Appendix~\ref{appendix:explicit-computations}, gives
\[
I_{12}\!\left(\frac{i}{\sqrt2},\frac{1}{\sqrt2}\right)
=
-\frac{11}{64}-\frac{i}{32}\neq0.
\]
Hence \(A_0|_{C_2}\) is not identically zero.

By \Cref{prop:partI-coefficient-functions}, \(A_0\) has left weight \(12\). By
\Cref{lem:partI-exceptional-fiber-period}, its restriction to an order-\(2\)
fiber descends with quotient frequency
\[
\frac{12}{2\cdot 2}=3.
\]
Therefore, after identifying \(C_2\) with a circle of period \(2\pi\), the
restriction \(A_0|_{C_2}\) is a nonzero complex exponential mode of frequency
\(3\). Its real part \(H_2|_{C_2}\) is therefore a nonconstant real
trigonometric mode of frequency \(3\), hence has exactly six nondegenerate
critical points on the quotient circle.
\end{proof}

\begin{lemma}
\label{lem:partI-H2-C3-C5}
Let
\[
H=\Re(A_0)\in E.
\]
Then:
\begin{enumerate}[label=\arabic*.]
\item \(H|_{C_5}\equiv 0\);
\item \(H|_{C_3}\) descends to a nonconstant real trigonometric mode of quotient
frequency \(2\);
\item consequently \(H|_{C_3}\) has exactly four nondegenerate critical points
on the quotient circle.
\end{enumerate}
\end{lemma}

\begin{proof}
By \Cref{prop:partI-coefficient-functions}, \(A_0\) has left weight \(12\).
Hence, by \Cref{lem:partI-exceptional-fiber-period}, its restriction to an
order-\(m\) exceptional fiber can descend nontrivially only if \(2m\mid 12\).

For \(m=5\), this divisibility fails, since \(10\nmid 12\). Therefore
\[
A_0|_{C_5}\equiv 0,
\]
and hence
\[
H|_{C_5}\equiv 0.
\]

For \(m=3\), the divisibility condition holds, since \(6\mid 12\), and the
quotient frequency is
\[
\frac{12}{2\cdot 3}=2.
\]
It remains to check that the restriction is not identically zero.

Take the order-\(3\) lift
\[
z_3=(r,se^{-i\pi/4}),
\qquad
r=\sqrt{\frac{1+1/\sqrt3}{2}},
\qquad
s=\sqrt{\frac{1-1/\sqrt3}{2}},
\]
as in Part~II. Then
\[
A_0(z_3)=I_{12}(z_3).
\]
A direct computation gives
\[
A_0(z_3)=\frac{11\sqrt3}{216}-\frac{i}{27}\neq 0.
\]
Hence \(A_0|_{C_3}\) is a nonzero complex quotient mode of frequency \(2\), and
its real part \(H|_{C_3}\) is a nonzero real trigonometric mode of frequency
\(2\). Such a mode has exactly four nondegenerate critical points on the circle.
\end{proof}

\begin{lemma}[Local splitting of a Morse--Bott critical circle]
\label{lem:MB-circle-splitting}
Let \(N\) be a smooth \(3\)-manifold, let \(C\subset N\) be an embedded circle,
and let \(f\in C^\infty(N)\) be Morse--Bott along \(C\). Assume that in tubular
coordinates \((\theta,x,y)\in S^1\times \R^2\) with \(C=\{x=y=0\}\), one has
\[
f(\theta,x,y)=c+Q(x,y)+R(\theta,x,y),
\qquad R(\theta,x,y)=O(\|(x,y)\|^3),
\]
where \(Q\) is a nondegenerate quadratic form in \((x,y)\).

Let \(h\in C^\infty(N)\), and assume that the restriction
\[
\theta\longmapsto h(\theta,0,0)
\]
is a Morse function on \(S^1\) with exactly \(k\) critical points.

Then, for all sufficiently small nonzero \(\varepsilon\), the perturbed function
\[
f_\varepsilon=f+\varepsilon h
\]
has exactly \(k\) critical points in a sufficiently small tubular neighborhood
of \(C\), all nondegenerate. Moreover, these critical points depend smoothly on
\(\varepsilon\) and converge to the corresponding critical points of
\(h|_C\) as \(\varepsilon\to0\).
\end{lemma}

\begin{proof}
Write local coordinates as in the statement:
\[
(\theta,x,y)\in S^1\times \R^2,
\qquad
C=\{x=y=0\},
\]
and set
\[
u=(x,y)\in\R^2.
\]
Thus
\[
f(\theta,u)=c+Q(u)+R(\theta,u),
\qquad
R(\theta,u)=O(\|u\|^3),
\]
with \(Q\) a nondegenerate quadratic form on \(\R^2\).

Let
\[
f_\varepsilon=f+\varepsilon h.
\]
We first solve the normal critical-point equations
\[
\partial_x f_\varepsilon(\theta,x,y)=0,
\qquad
\partial_y f_\varepsilon(\theta,x,y)=0.
\]

Set
\[
F(\theta,u,\varepsilon)=\nabla_u f_\varepsilon(\theta,u)
\in\R^2.
\]
Since
\[
f(\theta,u)=c+Q(u)+O(\|u\|^3),
\]
we have
\[
F(\theta,0,0)=0
\qquad\text{for all }\theta,
\]
and
\[
D_uF(\theta,0,0)=\Hess(Q),
\]
which is invertible because \(Q\) is nondegenerate. Therefore, by the implicit
function theorem, after shrinking to a sufficiently small tubular neighborhood
of \(C\) and restricting to sufficiently small \(|\varepsilon|\), there exists a
unique smooth map
\[
u=\xi(\theta,\varepsilon)\in\R^2,
\qquad
\xi(\theta,0)=0,
\]
such that
\[
\nabla_u f_\varepsilon(\theta,\xi(\theta,\varepsilon))=0
\]
for all \(\theta\) and all sufficiently small \(\varepsilon\). Moreover every
critical point of \(f_\varepsilon\) in that tubular neighborhood must lie on the
graph
\[
\Gamma_\varepsilon=\{(\theta,\xi(\theta,\varepsilon)):\theta\in S^1\}.
\]

Thus critical points of \(f_\varepsilon\) near \(C\) are exactly the critical
points of the reduced one-variable function
\[
g_\varepsilon(\theta)=
f_\varepsilon(\theta,\xi(\theta,\varepsilon)).
\]

We now compare \(g_\varepsilon\) with the circle-restriction \(h|_C\). Since
\(\xi(\theta,0)=0\) and \(D_u f(\theta,0)=0\), we have
\[
\xi(\theta,\varepsilon)=O(\varepsilon)
\]
uniformly in \(\theta\). Hence
\[
Q(\xi(\theta,\varepsilon))=O(\varepsilon^2),
\qquad
R(\theta,\xi(\theta,\varepsilon))=O(\varepsilon^3),
\]
and
\[
h(\theta,\xi(\theta,\varepsilon))
=
h(\theta,0,0)+O(\varepsilon)
\]
in \(C^2(S^1)\). Therefore
\[
g_\varepsilon(\theta)
=
c+\varepsilon h(\theta,0,0)+O(\varepsilon^2)
\]
in \(C^2(S^1)\). Equivalently,
\[
\varepsilon^{-1}(g_\varepsilon-c)\to h(\theta,0,0)
\qquad\text{in }C^2(S^1)
\]
as \(\varepsilon\to0\), for \(\varepsilon\neq0\).

By hypothesis, the function
\[
\theta\longmapsto h(\theta,0,0)
\]
is Morse on \(S^1\) with exactly \(k\) critical points. Since Morse critical
points on the circle are stable under \(C^2\)-small perturbations, it follows
that for all sufficiently small nonzero \(\varepsilon\), the reduced function
\(g_\varepsilon\) has exactly \(k\) critical points on \(S^1\), all
nondegenerate. These critical points depend smoothly on \(\varepsilon\) and
converge to the corresponding critical points of \(h|_C\) as \(\varepsilon\to0\).

It remains to show that the corresponding critical points of the full function
\(f_\varepsilon\) are nondegenerate. Let \(\theta_*\) be a critical point of
\(g_\varepsilon\), and set
\[
u_*=\xi(\theta_*,\varepsilon).
\]
Then \((\theta_*,u_*)\) is a critical point of \(f_\varepsilon\). At this point,
the \(uu\)-block of \(\Hess f_\varepsilon\) is
\[
D_u^2 f_\varepsilon(\theta_*,u_*),
\]
which remains invertible for small \(\varepsilon\), because it is a small
perturbation of the nondegenerate matrix \(\Hess(Q)\).

Now the graph \(u=\xi(\theta,\varepsilon)\) is precisely the critical manifold of
the normal equations, so by the Schur-complement formula recorded in
Appendix~\ref{appendix:schur-complement-reduction}, the second derivative of the
reduced function \(g_\varepsilon\) at \(\theta_*\) is
\[
g_\varepsilon''(\theta_*)
=
\partial_{\theta\theta} f_\varepsilon(\theta_*,u_*)
-
\partial_{\theta u} f_\varepsilon(\theta_*,u_*)
\bigl(D_u^2 f_\varepsilon(\theta_*,u_*)\bigr)^{-1}
\partial_{u\theta} f_\varepsilon(\theta_*,u_*).
\]
Since \(g_\varepsilon''(\theta_*)\neq0\), the full Hessian of \(f_\varepsilon\)
at \((\theta_*,u_*)\) is nondegenerate. Thus every critical point obtained in
this way is nondegenerate.

Therefore, for all sufficiently small nonzero \(\varepsilon\), the function
\(f_\varepsilon\) has exactly \(k\) critical points in a sufficiently small
tubular neighborhood of \(C\), all nondegenerate. These depend smoothly on
\(\varepsilon\) and converge to the corresponding critical points of \(h|_C\) as
\(\varepsilon\to0\).
\end{proof}

\begin{proposition}
\label{prop:partI-C2-six-points}
For all sufficiently small nonzero \(\varepsilon\), the function
\[
G_\varepsilon=F_0+\varepsilon H_2
\]
has exactly six nondegenerate critical points in a sufficiently small tubular
neighborhood of \(C_2\). Moreover, all six critical points have Morse index
\(1\) or \(2\).
\end{proposition}

\begin{proof}
By \Cref{cor:partI-C2-morse-bott}, \(F_0\) is Morse--Bott along \(C_2\), with
transverse Hessian nondegenerate of signature \((1,1)\). Therefore in a tubular
neighborhood of \(C_2\) one may choose local coordinates \((\theta,x,y)\), where
\(\theta\) parametrizes the quotient circle direction and \((x,y)\) are normal
coordinates, such that
\[
F_0(\theta,x,y)=c+Q(x,y)+R(\theta,x,y),
\]
where \(Q\) is a nondegenerate quadratic form of signature \((1,1)\) and
\[
R(\theta,x,y)=O(\|(x,y)\|^3).
\]

Restricting to the critical circle \(x=y=0\), we obtain
\[
G_\varepsilon(\theta,0,0)=c+\varepsilon h(\theta),
\]
where \(h(\theta)=H_2(\theta,0,0)\). By
\Cref{lem:partI-H2-restriction}, \(h\) has exactly six nondegenerate critical
points on the quotient circle.

By the Morse--Bott splitting lemma for circle critical manifolds
\Cref{lem:MB-circle-splitting}, for all sufficiently small
nonzero \(\varepsilon\), each of these six circle-critical points lifts uniquely
to a nearby nondegenerate critical point of the full function \(G_\varepsilon\),
and there are no others in a sufficiently small tubular neighborhood of \(C_2\).

Because the normal Hessian of \(F_0\) along \(C_2\) has signature \((1,1)\),
the normal contribution to the Morse index is \(1\). The circle direction
contributes either \(0\) or \(1\), according to the sign of \(h''\) at the
corresponding circle-critical point. Hence all six resulting critical points
have Morse index \(1\) or \(2\). In particular, none is a minimum or a maximum.
\end{proof}

\section{An explicit Morse function in \(E_{168}(g_0)\) with at least eight critical points}
\label{sec:partI-eight-critical-points}

\begin{proposition}
\label{prop:partI-explicit-morse}
There exists a function
\[
F\in E_{168}(g_0)
\]
which is Morse and has at least eight critical points. In fact, one can choose
\(F\) with exactly twelve critical points.
\end{proposition}

\begin{proof}
Set
\[
H=\Re(A_0)\in E,
\qquad
K=\Re(A_1)\in E.
\]
By \Cref{lem:partI-H2-restriction}, the restriction of \(H\) to the order-\(2\)
exceptional fiber \(C_2\) descends to a nonzero real trigonometric mode of
quotient frequency \(3\), hence has exactly six nondegenerate critical points on
the quotient circle. By \Cref{lem:partI-H2-C3-C5}, the restriction of \(H\) to
\(C_3\) descends to a nonzero real trigonometric mode of quotient frequency \(2\),
hence has exactly four nondegenerate critical points on the quotient circle, and
\(H|_{C_5}\equiv 0\).

On the other hand, \(A_1\) has left weight \(10\), so by
\Cref{lem:partI-exceptional-fiber-period} its restriction vanishes on \(C_3\)
and \(C_2\), while on \(C_5\) it descends with quotient frequency
\[
\frac{10}{2\cdot 5}=1.
\]
Since \(A_1(z_5)=1\neq 0\) at the order-\(5\) lift \(z_5=(1,0)\), its
restriction to \(C_5\) is nontrivial. Therefore
\[
K|_{C_5}
\]
is a nonzero real trigonometric mode of quotient frequency \(1\), and hence has
exactly two nondegenerate critical points on the quotient circle.

By \Cref{cor:partI-F0-morse-bott-all}, the seed function \(F_0\) is Morse--Bott
along each of the three critical circles
\[
C_5,\qquad C_3,\qquad C_2.
\]
Choose pairwise disjoint tubular neighborhoods
\[
U_5,\qquad U_3,\qquad U_2
\]
of these circles, so small that the Morse--Bott normal form holds in each one.
Let
\[
L=M\setminus (U_5\cup U_3\cup U_2).
\]
Since \(L\) is compact and contains no critical point of \(F_0\), there exists
\(c>0\) such that
\[
|\nabla F_0|_{g_0}\ge c
\qquad\text{on }L.
\]
Because \(H,K\in E\) and \(E\) is finite-dimensional, for sufficiently small
\((a,b)\in\R^2\) one has
\[
|\nabla(aH+bK)|_{g_0}\le \frac c2
\qquad\text{on }L.
\]
Hence for all sufficiently small \((a,b)\),
\[
|\nabla(F_0+aH+bK)|_{g_0}\ge \frac c2>0
\qquad\text{on }L,
\]
so no critical point lies outside \(U_5\cup U_3\cup U_2\).

Now consider
\[
F_{a,b}=F_0+aH+bK.
\]

On \(C_5\), one has \(H|_{C_5}\equiv 0\), so
\[
F_{a,b}|_{C_5}=F_0|_{C_5}+bK|_{C_5}.
\]
Since \(K|_{C_5}\) is a nonzero frequency-\(1\) mode, it has exactly two
nondegenerate critical points on the quotient circle. By the Morse--Bott
splitting lemma \Cref{lem:MB-circle-splitting}, for all
sufficiently small nonzero \(b\), the function \(F_{a,b}\) has exactly two
nondegenerate critical points in \(U_5\).

On \(C_3\), one has \(K|_{C_3}\equiv 0\), so
\[
F_{a,b}|_{C_3}=F_0|_{C_3}+aH|_{C_3}.
\]
Since \(H|_{C_3}\) is a nonzero frequency-\(2\) mode, it has exactly four
nondegenerate critical points on the quotient circle. Again by the Morse--Bott
splitting lemma, for all sufficiently small nonzero \(a\), the function
\(F_{a,b}\) has exactly four nondegenerate critical points in \(U_3\).

On \(C_2\), one also has \(K|_{C_2}\equiv 0\), so
\[
F_{a,b}|_{C_2}=F_0|_{C_2}+aH|_{C_2}.
\]
By \Cref{lem:partI-H2-restriction}, \(H|_{C_2}\) is a nonzero frequency-\(3\)
mode, hence has exactly six nondegenerate critical points on the quotient
circle. The Morse--Bott splitting lemma therefore yields exactly six
nondegenerate critical points of \(F_{a,b}\) in \(U_2\), for all sufficiently
small nonzero \(a\).

Therefore, for all sufficiently small nonzero \(a,b\),
\[
\#\Crit(F_{a,b})=2+4+6=12.
\]
All these critical points are nondegenerate, so \(F_{a,b}\) is Morse. Taking
\[
F=F_{a,b}
\]
for such a choice of parameters proves the claim.
\end{proof}

\part{A perturbative existence theorem beyond the spherical metric}
\label{part:perturbative-program}

This part develops the perturbative mechanism beyond the spherical metric on
\[
M=S^3/I^*.
\]
Its logic is as follows. We first prove an abstract reduction principle: if a
nearby metric has simple first eigenvalue and minimal Morse first eigenfunction,
then it has property \(P\). We then derive the first-order conformal splitting
operator on the spherical first eigenspace
\[
E=E_{168}(g_0),
\qquad g_0=g_{\mathrm{sph}},
\]
and construct a canonical finite-dimensional realizable splitting space
\[
\mathscr B\subset \Sym(E).
\]
Inside \(\mathscr B\) we exhibit a distinguished seed operator \(A_0\) whose
simple lowest eigendirection is precisely the canonical invariant line
\(\mathbb RF_0\subset E\), and we prove that the lowest-eigenline map is a
submersion at \(A_0\). It follows that every projective line sufficiently close
to \([F_0]\) is realized as the simple lowest eigendirection of some operator in
\(\mathscr B\).

We then prove the remaining geometric input: among the explicit nearby lines
generated by
\[
F_0,\qquad \Re(F_1),\qquad \Re(F_3),\qquad \Re(F_4),
\]
there exist minimal-Morse lines arbitrarily close to \([F_0]\). Combining this
with the branch-selection theorem yields the perturbative existence theorem
beyond the spherical metric.

\section{Reduction to the simple first-eigenfunction regime}
\label{sec:partII-simple-regime}

We begin with the abstract heat-flow mechanism.

\begin{theorem}[Simple first-eigenfunction reduction]
\label{thm:partII-simple-first-eigenfunction-implies-P}
Let \((N,g)\) be a smooth closed connected Riemannian manifold. Assume that:
\begin{enumerate}[label=\arabic*.]
\item the first nonzero eigenspace \(E_{\lambda_1}(g)\) is one-dimensional;
\item an \(L^2(N,g)\)-normalized first eigenfunction \(\phi_1\) is minimal
Morse.
\end{enumerate}
Then \((N,g)\) has property \(P\).
\end{theorem}

\begin{proof}
Since \(E_{\lambda_1}(g)\) is one-dimensional, every nonzero element of
\(E_{\lambda_1}(g)\) is of the form \(c\phi_1\) with \(c\neq0\). Multiplication
by a nonzero scalar does not change the critical set, Morse character, or the
number of critical points, so every nonzero vector in \(E_{\lambda_1}(g)\) is
again minimal Morse.

Let
\[
\Pi_1:L^2(N,g)\to E_{\lambda_1}(g)
\]
be the orthogonal projection, and define
\[
\mathcal S=\{f\in L^2(N,g):\Pi_1f\neq0\}.
\]
Since \(\ker \Pi_1\) is a closed proper subspace, \(\mathcal S\) is open and
dense in \(L^2(N,g)\).

Fix \(f\in\mathcal S\), and let \(u(\cdot,t)\) be the heat evolution with
initial condition \(f\). Write
\[
f=f_0+c\phi_1+f_\perp,
\]
where \(f_0\in E_{\lambda_0}(g)\), \(c\neq0\), and
\[
f_\perp\in \overline{\bigoplus_{k\ge2}E_{\lambda_k}(g)}.
\]
Then
\[
u(\cdot,t)=f_0+ce^{-\lambda_1t}\phi_1+\sum_{k\ge2}e^{-\lambda_kt}f_k,
\]
hence
\[
e^{\lambda_1t}\bigl(u(\cdot,t)-f_0\bigr)\longrightarrow c\phi_1
\qquad\text{in }C^\infty(N)
\]
as \(t\to\infty\), in particular in \(C^2(N)\).

Because \(c\phi_1\) is minimal Morse, \(C^2\)-stability of nondegenerate
critical points gives a \(C^2\)-neighborhood \(\mathcal U\) of \(c\phi_1\)
such that every \(h\in\mathcal U\) is minimal Morse. For all sufficiently
large \(t\),
\[
e^{\lambda_1t}\bigl(u(\cdot,t)-f_0\bigr)\in\mathcal U.
\]
Subtracting the constant \(f_0\) and multiplying by the positive scalar
\(e^{-\lambda_1t}\) do not change the critical set or Morse character.
Therefore \(u(\cdot,t)\) is minimal Morse for all sufficiently large \(t\).

Thus every \(f\in\mathcal S\) satisfies the defining property of \(P\), and
\(\mathcal S\) is open and dense. Therefore \((N,g)\) has property \(P\).
\end{proof}

The perturbative step also requires stability of a simple first eigenfunction.

\begin{theorem}[Stability of a simple Morse first eigenfunction]
\label{thm:partII-first-eigenfunction-stability}
Let \(N\) be a smooth closed connected manifold, and let \(g_*\) be a smooth
Riemannian metric on \(N\). Assume that:
\begin{enumerate}[label=\arabic*.]
\item \(\lambda_1(g_*)\) is simple;
\item an \(L^2(N,g_*)\)-normalized first eigenfunction \(\phi_*\) is Morse
with exactly \(m\) critical points.
\end{enumerate}
Then there exists a \(C^\infty\)-neighborhood \(\mathcal U\) of \(g_*\) such
that, for every \(h\in\mathcal U\),
\begin{enumerate}[label=\arabic*.]
\item \(\lambda_1(h)\) is simple;
\item the corresponding normalized first eigenfunction \(\phi_h\) is Morse;
\item \(\phi_h\) has exactly \(m\) critical points.
\end{enumerate}
\end{theorem}

\begin{proof}
Simplicity of \(\lambda_1(g_*)\) implies, by analytic perturbation theory for
self-adjoint elliptic operators, that after fixing sign continuously the
normalized first eigenfunction depends smoothly on the metric in a
\(C^\infty\)-neighborhood of \(g_*\); see \cite[Chap.~VII]{Kato}. In
particular,
\[
\phi_h\to \phi_*
\qquad\text{in }C^\infty(N)
\]
as \(h\to g_*\), hence also in \(C^2(N)\).

Let
\[
\Crit(\phi_*)=\{p_1,\dots,p_m\}.
\]
Choose pairwise disjoint coordinate balls
\[
U_1,\dots,U_m
\]
with \(p_j\in U_j\), so small that \(\phi_*\) has exactly one critical point in
each \(U_j\), namely \(p_j\), and that point is nondegenerate. By
\(C^2\)-stability of nondegenerate critical points, if \(h\) is sufficiently
close to \(g_*\), then \(\phi_h\) has exactly one nondegenerate critical point
in each \(U_j\).

Now let
\[
K=N\setminus \bigcup_{j=1}^m U_j.
\]
Since \(K\) is compact and contains no critical point of \(\phi_*\), there is
\(c>0\) such that
\[
|\nabla \phi_*|_{g_*}\ge c
\qquad\text{on }K.
\]
Because \(\phi_h\to\phi_*\) in \(C^1\), after shrinking the neighborhood if
necessary we have
\[
|\nabla \phi_h|_{g_*}\ge \frac c2>0
\qquad\text{on }K.
\]
Therefore \(\phi_h\) has no critical points on \(K\). Hence \(\phi_h\) has
exactly one critical point in each \(U_j\) and no others, so it is Morse with
exactly \(m\) critical points. Openness of simplicity for isolated
self-adjoint eigenvalues yields the first assertion.
\end{proof}

For the Poincar\'e dodecahedral space, Part~I identifies minimality
numerically.

\begin{corollary}
\label{cor:partII-six-critical-points}
Let \(M=S^3/I^*\). If a Morse function on \(M\) has exactly six critical
points, then it is minimal.
\end{corollary}

\begin{proof}
This is \Cref{prop:partI-minimal-six}.
\end{proof}

\section{First-order conformal splitting on the spherical first eigenspace}
\label{sec:partII-spherical-splitting}

We now specialize to the spherical metric \(g_0=g_{\mathrm{sph}}\) on
\[
M=S^3/I^*.
\]
Let
\[
\lambda_{\mathrm{sph}}=\lambda_1(g_0)=168,
\qquad
E=E_{\lambda_{\mathrm{sph}}}(g_0).
\]
By Part~I, \(\dim E=13\).

Consider a conformal perturbation
\[
g_\varepsilon=e^{2\varepsilon\rho}g_0,
\qquad
\rho\in C^\infty(M).
\]
Fix a real \(L^2(M,g_0)\)-orthonormal basis
\[
\phi_0,\dots,\phi_{12}
\]
of \(E\).

For \(u,v\in C^\infty(M)\), define
\[
a_\varepsilon(u,v)
=
\int_M \langle \nabla u,\nabla v\rangle_{g_\varepsilon}\,dV_{g_\varepsilon},
\qquad
m_\varepsilon(u,v)
=
\int_M uv\,dV_{g_\varepsilon}.
\]
The eigenvalue problem for \(-\Delta_{g_\varepsilon}\) is equivalent to
\[
a_\varepsilon(u,v)=\lambda(\varepsilon)\,m_\varepsilon(u,v)
\qquad
\forall v\in C^\infty(M).
\]

Since \(\dim M=3\),
\[
dV_{g_\varepsilon}=e^{3\varepsilon\rho}\,dV_{g_0},
\qquad
\langle \nabla u,\nabla v\rangle_{g_\varepsilon}
=
e^{-2\varepsilon\rho}\langle \nabla u,\nabla v\rangle_{g_0},
\]
so
\[
a_\varepsilon(u,v)
=
\int_M e^{\varepsilon\rho}\langle \nabla u,\nabla v\rangle_{g_0}\,dV_{g_0},
\qquad
m_\varepsilon(u,v)
=
\int_M e^{3\varepsilon\rho}uv\,dV_{g_0}.
\]
Differentiating at \(\varepsilon=0\),
\[
a_0'(u,v)
=
\int_M \rho\,\langle \nabla u,\nabla v\rangle_{g_0}\,dV_{g_0},
\qquad
m_0'(u,v)
=
3\int_M \rho\,uv\,dV_{g_0}.
\]

\begin{proposition}[First-order splitting formula]
\label{prop:partII-splitting-formula}
For every \(\rho\in C^\infty(M)\), the first-order splitting operator on
\[
E=E_{\lambda_{\mathrm{sph}}}(g_0)
\]
is represented in the basis \(\phi_0,\dots,\phi_{12}\) by the symmetric matrix
\[
B^{(\rho)}=(B_{ij}^{(\rho)})_{0\le i,j\le 12},
\qquad
B_{ij}^{(\rho)}
=
a_0'(\phi_i,\phi_j)-\lambda_{\mathrm{sph}}\,m_0'(\phi_i,\phi_j),
\]
and one has the explicit formula
\[
B_{ij}^{(\rho)}
=
-2\lambda_{\mathrm{sph}}\int_M \rho\,\phi_i\phi_j\,dV_{g_0}
-\frac12\int_M (\Delta_{g_0}\rho)\,\phi_i\phi_j\,dV_{g_0}.
\]
Equivalently, if
\[
q_\rho
=
2\lambda_{\mathrm{sph}}\rho+\frac12\Delta_{g_0}\rho,
\]
then
\[
B_{ij}^{(\rho)}
=
-\int_M q_\rho\,\phi_i\phi_j\,dV_{g_0}.
\]
\end{proposition}

\begin{proof}
Since each \(\phi_j\in E\) satisfies
\[
\Delta_{g_0}\phi_j=\lambda_{\mathrm{sph}}\phi_j,
\]
one computes, using the identity
\[
\Delta_{g_0}(\rho\phi_i\phi_j)
=
(\Delta_{g_0}\rho)\phi_i\phi_j
+\rho\,\Delta_{g_0}(\phi_i\phi_j)
+2\langle \nabla\rho,\nabla(\phi_i\phi_j)\rangle_{g_0},
\]
integration over the closed manifold \(M\), and the eigenvalue equations for
\(\phi_i,\phi_j\), that
\[
\int_M \rho\,\langle \nabla\phi_i,\nabla\phi_j\rangle_{g_0}\,dV_{g_0}
=
\lambda_{\mathrm{sph}}\int_M \rho\,\phi_i\phi_j\,dV_{g_0}
-\frac12\int_M (\Delta_{g_0}\rho)\,\phi_i\phi_j\,dV_{g_0}.
\]
Substituting into the definition of \(B_{ij}^{(\rho)}\) gives the formula.
\end{proof}

\section{The realizable splitting space}
\label{sec:partII-realizable-splitting-space}

For the first-order splitting problem, it is convenient to work with the
product space
\[
\mathcal P=\operatorname{span}\{fg:f,g\in E\}\subset L^2(M,g_0).
\]
For each \(q\in C^\infty(M)\), define the symmetric operator \(B(q)\in\Sym(E)\)
by
\[
\langle B(q)f,g\rangle_{L^2(M,g_0)}
=
-\int_M q\,fg\,dV_{g_0}
\qquad (f,g\in E).
\]
Since only the pairings with products \(fg\) matter, the operator \(B(q)\)
depends only on the \(L^2(M,g_0)\)-projection of \(q\) onto \(\mathcal P\).
Hence the image of the map \(q\mapsto B(q)\) is the finite-dimensional space
\[
\mathscr B=B(\mathcal P)\subset \Sym(E).
\]

\begin{remark}
By \Cref{prop:partII-splitting-formula}, if
\[
q_\rho=2\lambda_{\mathrm{sph}}\rho+\frac12\Delta_{g_0}\rho,
\]
then the first-order splitting operator associated with the conformal factor
\(\rho\) is precisely
\[
B^{(\rho)}=B(q_\rho)=B(\Pi_{\mathcal P}q_\rho)\in\mathscr B,
\]
where \(\Pi_{\mathcal P}\) denotes the \(L^2(M,g_0)\)-orthogonal projection
onto \(\mathcal P\). Thus \(\mathscr B\) is exactly the finite-dimensional
operator space generated by first-order conformal splittings.
\end{remark}

\begin{lemma}
\label{lem:partII-B-on-P-injective}
The restriction
\[
B|_{\mathcal P}:\mathcal P\to \mathscr B
\]
is injective.
\end{lemma}

\begin{proof}
Assume \(q\in\mathcal P\) and \(B(q)=0\). Then for every \(f,g\in E\),
\[
0=\langle B(q)f,g\rangle_{L^2(M,g_0)}
=
-\int_M q\,fg\,dV_{g_0}.
\]
Since
\[
\mathcal P=\operatorname{span}\{fg:f,g\in E\},
\]
it follows that
\[
\int_M q\,h\,dV_{g_0}=0
\qquad\forall h\in\mathcal P.
\]
Thus \(q\) is \(L^2(M,g_0)\)-orthogonal to \(\mathcal P\). Because \(q\in\mathcal P\),
we conclude that \(q=0\). Hence \(B|_{\mathcal P}\) is injective.
\end{proof}

\begin{proposition}[Every operator in \(\mathscr B\) is conformally realizable]
\label{prop:partII-every-B-realizable-by-rho}
For every operator
\[
A\in \mathscr B,
\]
there exists a smooth function
\[
\rho\in C^\infty(M)
\]
such that
\[
B^{(\rho)}=A.
\]
Equivalently, the map
\[
C^\infty(M)\longrightarrow \mathscr B,
\qquad
\rho\longmapsto B^{(\rho)},
\]
is surjective.

Moreover, if \(q_A\in\mathcal P\) denotes the unique element satisfying
\[
A=B(q_A),
\]
then there exists a unique smooth function
\[
\rho_A\in C^\infty(M)
\]
such that
\[
2\lambda_{\mathrm{sph}}\rho_A+\frac12\Delta_{g_0}\rho_A=q_A.
\]
For this canonical choice one has
\[
B^{(\rho_A)}=A.
\]
\end{proposition}

\begin{proof}
Recall from \Cref{prop:partII-splitting-formula} that
\[
B^{(\rho)}=B(q_\rho),
\qquad
q_\rho=2\lambda_{\mathrm{sph}}\rho+\frac12\Delta_{g_0}\rho.
\]

Let
\[
A\in\mathscr B.
\]
By definition of \(\mathscr B=B(\mathcal P)\), there exists \(q_A\in\mathcal P\)
such that
\[
A=B(q_A).
\]
By \Cref{lem:partII-B-on-P-injective}, this \(q_A\) is unique. Since
\(\mathcal P\) is spanned by products of smooth eigenfunctions in \(E\), every
element of \(\mathcal P\) is smooth. In particular,
\[
q_A\in C^\infty(M).
\]

Consider the elliptic operator
\[
L=\frac12\Delta_{g_0}+2\lambda_{\mathrm{sph}}.
\]
Since \(\Delta_{g_0}\) is self-adjoint with nonnegative spectrum in our sign
convention, every eigenvalue of \(L\) has the form
\[
\frac12\mu+2\lambda_{\mathrm{sph}},
\qquad
\mu\in\Spec(\Delta_{g_0}),
\]
and is therefore strictly positive. Hence \(0\notin\Spec(L)\), so
\[
L:C^\infty(M)\to C^\infty(M)
\]
is an isomorphism.

Therefore there exists a unique smooth function \(\rho_A\) satisfying
\[
L\rho_A=q_A,
\qquad\text{i.e.}\qquad
2\lambda_{\mathrm{sph}}\rho_A+\frac12\Delta_{g_0}\rho_A=q_A.
\]
Thus
\[
q_{\rho_A}=q_A.
\]
Applying \Cref{prop:partII-splitting-formula},
\[
B^{(\rho_A)}=B(q_{\rho_A})=B(q_A)=A.
\]
This proves the asserted surjectivity and the uniqueness of the canonical choice
\(\rho_A\) satisfying \(q_{\rho_A}=q_A\).
\end{proof}

\medskip

\noindent\textbf{The seed operator.}

Fix the base point
\[
o=eI^*\in M=SU(2)/I^*.
\]

Since \(\mathcal P\subset L^2(M,g_0)\) is finite-dimensional, the evaluation
functional at \(o\),
\[
\operatorname{ev}_o:\mathcal P\to\R,
\qquad
\operatorname{ev}_o(h)=h(o),
\]
is continuous. By the Riesz representation theorem, there exists a unique
element
\[
q_o\in\mathcal P
\]
such that
\[
h(o)=\int_M q_o\,h\,dV_{g_0}
\qquad\forall h\in\mathcal P.
\]

Define the associated operator
\[
A_0=B(q_o)\in\mathscr B.
\]

\begin{proposition}
\label{prop:partII-A0-rayleigh}
For every \(f\in E\),
\[
\langle A_0 f,f\rangle_{L^2(M,g_0)}=-f(o)^2.
\]
\end{proposition}

\begin{proof}
By definition of \(B(q_o)\),
\[
\langle A_0 f,f\rangle
=
-\int_M q_o\,f^2\,dV_{g_0}.
\]
Since \(f^2\in\mathcal P\), the defining property of \(q_o\) gives
\[
\int_M q_o\,f^2\,dV_{g_0}=f^2(o).
\]
Hence
\[
\langle A_0 f,f\rangle=-f(o)^2.
\]
\end{proof}

\medskip

\noindent\textbf{The reproducing kernel vector.}

Since \(E\subset C^\infty(M)\) is finite-dimensional, evaluation at \(o\)
defines a continuous linear functional
\[
\operatorname{ev}_o:E\to\R.
\]
By the Riesz representation theorem, there exists a unique vector
\[
K_o\in E
\]
such that
\[
f(o)=\langle f,K_o\rangle_{L^2(M,g_0)}
\qquad\forall f\in E.
\]

\begin{theorem}
\label{thm:partII-A0-simple-lowest}
The operator \(A_0\) has simple lowest eigenvalue, and its lowest
eigendirection is the line \(\mathbb RK_o\).
\end{theorem}

\begin{proof}
By \Cref{prop:partII-A0-rayleigh},
\[
\langle A_0f,f\rangle=-f(o)^2
\qquad\text{for every unit }f\in E.
\]
By the reproducing-kernel identity,
\[
f(o)=\langle f,K_o\rangle,
\]
hence
\[
f(o)^2\le \|K_o\|^2
\]
for every unit \(f\), with equality if and only if
\[
f=\pm \frac{K_o}{\|K_o\|}.
\]
Therefore the Rayleigh quotient of \(A_0\) is uniquely minimized on the line
\(\mathbb RK_o\). Since \(A_0\) is symmetric, this line is its lowest
eigendirection and the lowest eigenvalue is simple.
\end{proof}

\begin{proposition}
\label{prop:partII-Ko-equals-F0-line}
Let \(F_0\in E\) be a nonzero generator of the unique \(I^*\)-invariant line in
\(E\). Then
\[
\mathbb RK_o=\mathbb RF_0.
\]
\end{proposition}

\begin{proof}
Write
\[
G=SU(2),\qquad H=I^*,
\qquad M=G/H.
\]
The left action of \(G\) on \(M\) preserves the spherical metric \(g_0\), hence
induces an orthogonal representation on \(E\), given by
\[
(L_a f)(x)=f(a^{-1}x),
\qquad a\in G.
\]

Since \(o=eH\), every \(h\in H\) fixes \(o\). Therefore the evaluation
functional at \(o\),
\[
\operatorname{ev}_o:E\to\R,
\qquad
\operatorname{ev}_o(f)=f(o),
\]
is \(H\)-invariant:
\[
\operatorname{ev}_o(L_hf)=(L_hf)(o)=f(h^{-1}o)=f(o)
\qquad\forall h\in H,\ \forall f\in E.
\]

Now let \(h\in H\). For every \(f\in E\),
\[
\langle f,L_hK_o\rangle
=
\langle L_{h^{-1}}f,K_o\rangle
=
(L_{h^{-1}}f)(o)
=
f(o)
=
\langle f,K_o\rangle.
\]
Hence
\[
L_hK_o=K_o
\qquad\forall h\in H.
\]
Thus \(K_o\in E^H\).

By the representation-theoretic description of \(E\), the \(H=I^*\)-fixed
subspace \(E^H\) is one-dimensional; it is precisely the distinguished
\(I^*\)-invariant line generated by \(F_0\). Therefore
\[
K_o=cF_0
\]
for some \(c\in\R\).

Finally, \(K_o\neq0\), since
\[
K_o(o)=\langle K_o,K_o\rangle_{L^2(M,g_0)}=\|K_o\|_{L^2(M,g_0)}^2>0.
\]
Thus \(c\neq0\), and so
\[
\mathbb RK_o=\mathbb RF_0.
\]
\end{proof}

\begin{corollary}
\label{cor:partII-A0-simple-lowest-F0}
The operator \(A_0\) has simple lowest eigenvalue, and its lowest
eigendirection is the line \(\mathbb RF_0\).
\end{corollary}

\begin{proof}
Combine \Cref{thm:partII-A0-simple-lowest} with
\Cref{prop:partII-Ko-equals-F0-line}.
\end{proof}

Let
\[
\Omega=\{A\in\mathscr B:\lambda_1(A)<\lambda_2(A)\}
\]
be the simple-lowest locus in \(\mathscr B\). For \(A\in\Omega\), denote by
\[
\ell(A)\in\mathbb P(E)
\]
its lowest eigenline.

\begin{proposition}
\label{prop:partII-differential-eigenline-map}
Let \(A\in\Omega\), and let \(v\in E\) be a unit lowest eigenvector,
\[
Av=\lambda_1(A)v.
\]
Identifying
\[
T_{\ell(A)}\mathbb P(E)\cong v^\perp,
\]
the differential of \(\ell\) at \(A\) is given by
\[
d\ell_A(H)
=
-(A-\lambda_1(A)I)^{-1}P_{v^\perp}(Hv),
\qquad H\in\mathscr B.
\]
In particular, \(d\ell_A\) is surjective if and only if
\[
P_{v^\perp}(\mathscr Bv)=v^\perp.
\]
\end{proposition}

\begin{proof}
See Appendix~\ref{appendix:eigenline-differential}.
\end{proof}

\begin{proposition}
\label{prop:partII-adjoint-multiplication}
Let \(v\in E\) be a unit vector, and define
\[
T_v:\mathcal P\to v^\perp,
\qquad
T_v(q)=P_{v^\perp}(B(q)v).
\]
Then for every \(q\in\mathcal P\) and every \(w\in v^\perp\),
\[
\langle T_v(q),w\rangle_{L^2(M,g_0)}
=
-\int_M q\,vw\,dV_{g_0}.
\]
Consequently, the adjoint of \(T_v\) is, up to sign, the multiplication map
\[
\mu_v:v^\perp\to\mathcal P,
\qquad
\mu_v(w)=vw.
\]
In particular, \(T_v\) is surjective if and only if \(\mu_v\) is injective.
\end{proposition}

\begin{proof}
For \(q\in\mathcal P\) and \(w\in v^\perp\),
\[
\langle T_v(q),w\rangle
=
\langle B(q)v,w\rangle
=
-\int_M q\,vw\,dV_{g_0}.
\]
This identifies the adjoint of \(T_v\) with \(-\mu_v\). Since both spaces are
finite-dimensional, surjectivity of \(T_v\) is equivalent to injectivity of its
adjoint.
\end{proof}

\begin{lemma}
\label{lem:partII-multiplication-Ko-injective}
The multiplication map
\[
\mu_{K_o}:K_o^\perp\to \mathcal P,
\qquad
w\mapsto K_ow
\]
is injective.
\end{lemma}

\begin{proof}
Suppose \(w\in K_o^\perp\) satisfies \(K_ow\equiv0\). Since \(K_o\) is a nonzero
eigenfunction, its zero set has empty interior. Hence there is a nonempty open
set on which \(K_o\neq0\), and on that open set \(w=0\). Since \(w\) is also an
eigenfunction, unique continuation implies \(w\equiv0\).
\end{proof}

\begin{theorem}
\label{thm:partII-local-submersion}
The differential
\[
d\ell_{A_0}:\mathscr B\to T_{[F_0]}\mathbb P(E)
\]
is surjective. Equivalently, \(\ell\) is a submersion at \(A_0\).
\end{theorem}

\begin{proof}
By \Cref{cor:partII-A0-simple-lowest-F0}, the lowest eigendirection of \(A_0\)
is \(\mathbb RF_0=\mathbb RK_o\). By
\Cref{prop:partII-differential-eigenline-map}, it suffices to prove that
\[
P_{K_o^\perp}(\mathscr BK_o)=K_o^\perp.
\]
By \Cref{prop:partII-adjoint-multiplication}, this is equivalent to
injectivity of
\[
\mu_{K_o}:K_o^\perp\to\mathcal P,
\]
which is exactly \Cref{lem:partII-multiplication-Ko-injective}.
\end{proof}

\begin{corollary}
\label{cor:partII-realizable-neighborhood}
There exists an open neighborhood
\[
U\subset \mathbb P(E)
\]
of \([F_0]\) such that
\[
U\subset \ell(\Omega).
\]
Equivalently, every projective line in \(U\) is realized as the simple lowest
eigendirection of some operator in \(\mathscr B\).
\end{corollary}

\begin{proof}
Since \(\ell\) is a smooth map between finite-dimensional manifolds and is a
submersion at \(A_0\), its image contains a neighborhood of
\[
\ell(A_0)=[F_0].
\]
\end{proof}

\section{Passage from a selected splitter to an actual nearby metric}
\label{sec:partII-branch-theorem}

We now make precise the passage from a first-order splitter
\(B^{(\rho)}\in\mathscr B\) to an actual nearby branch of the Laplacian.

Let
\[
g_\varepsilon=e^{2\varepsilon\rho}g_0,
\qquad \rho\in C^\infty(M),
\]
and write
\[
L_\varepsilon=-\Delta_{g_\varepsilon}
\]
as an unbounded self-adjoint operator on \(L^2(M,g_\varepsilon)\). Equivalently,
after transporting everything to the fixed Hilbert space \(H_0=L^2(M,g_0)\)
through the unitary identification
\[
U_\varepsilon:L^2(M,g_\varepsilon)\to H_0,
\qquad
U_\varepsilon f=e^{\frac32\varepsilon\rho}f,
\]
we obtain a smooth family of self-adjoint elliptic operators
\[
\widetilde L_\varepsilon=U_\varepsilon L_\varepsilon U_\varepsilon^{-1}
\]
on \(H_0\), with
\[
\widetilde L_0=L_0=-\Delta_{g_0}.
\]
All spectral statements below are taken on this fixed Hilbert space.

Let
\[
0=\mu_0<\mu_1=168<\mu_2
\]
be the first three distinct eigenvalues of \(L_0\), and let
\[
E=\ker(L_0-168I).
\]

\begin{proposition}[Isolated cluster and transported finite-dimensional reduction]
\label{prop:partII-cluster-reduction}
There exist real numbers
\[
\gamma_-<168<\gamma_+,
\qquad
(0,\gamma_-)\cap\Spec(L_0)=\varnothing,
\qquad
(\gamma_-,\gamma_+)\cap\Spec(L_0)=\{168\},
\]
and \(\varepsilon_0>0\) such that, for all \(|\varepsilon|<\varepsilon_0\), the
following hold:

\begin{enumerate}[label=\arabic*.]
\item The spectrum of \(\widetilde L_\varepsilon\) in \((\gamma_-,\gamma_+)\)
consists of exactly \(13\) eigenvalues, counted with multiplicity, and there is
no positive spectrum of \(\widetilde L_\varepsilon\) in \((0,\gamma_-]\).

\item The Riesz spectral projection
\[
P_\varepsilon
=
\frac{1}{2\pi i}\int_\Gamma (z-\widetilde L_\varepsilon)^{-1}\,dz,
\]
where \(\Gamma\) is a positively oriented circle enclosing \(168\) and no other
point of \(\Spec(L_0)\), is well-defined for \(|\varepsilon|<\varepsilon_0\),
has rank \(13\), and depends smoothly on \(\varepsilon\) in operator norm.

\item There exists a smooth family of unitary isomorphisms
\[
W_\varepsilon:E\to E_\varepsilon=\operatorname{Ran}(P_\varepsilon),
\qquad
W_0=\Id_E.
\]

\item The transported finite-dimensional operator
\[
T_\varepsilon=W_\varepsilon^{-1}\,\widetilde L_\varepsilon|_{E_\varepsilon}\,W_\varepsilon
\in \Sym(E)
\]
depends smoothly on \(\varepsilon\), satisfies
\[
T_0=168\,I_E,
\]
and has first-order expansion
\[
T_\varepsilon=168\,I_E+\varepsilon B^{(\rho)}+o(\varepsilon)
\qquad\text{in operator norm on }E.
\]
\end{enumerate}
\end{proposition}

\begin{proof}
Choose \(\gamma_-,\gamma_+\) with
\[
0<\gamma_-<168<\gamma_+<\mu_2.
\]
Then \(168\) is the only eigenvalue of \(L_0\) in \((\gamma_-,\gamma_+)\), and
there is no positive spectrum in \((0,\gamma_-]\).

Because \(\widetilde L_\varepsilon\) is a smooth family of self-adjoint elliptic
operators on the fixed Hilbert space \(H_0\), standard perturbation theory
implies that, for \(|\varepsilon|\) sufficiently small, the spectrum near \(168\)
remains separated from the rest of the spectrum, the associated Riesz projection
\[
P_\varepsilon=\frac{1}{2\pi i}\int_\Gamma (z-\widetilde L_\varepsilon)^{-1}\,dz
\]
is well-defined, has constant rank \(13\), and depends smoothly on
\(\varepsilon\); see \cite[Chap.~VII]{Kato}. This proves (1) and (2), after
possibly shrinking \(\varepsilon_0\).

By Kato's canonical spectral transport, after shrinking \(\varepsilon_0\) if
necessary there exists a smooth family of unitaries
\[
W_\varepsilon:H_0\to H_0,
\qquad
W_0=\Id_{H_0},
\qquad
W_\varepsilon(E)=E_\varepsilon.
\]
Restricting to \(E\) gives a smooth family of unitary isomorphisms
\[
W_\varepsilon:E\to E_\varepsilon,
\qquad
W_0=\Id_E.
\]
This proves (3).

Now define
\[
T_\varepsilon=W_\varepsilon^{-1}\,\widetilde L_\varepsilon|_{E_\varepsilon}\,W_\varepsilon
\in \Sym(E).
\]
Since \(W_\varepsilon\) and \(\widetilde L_\varepsilon\) depend smoothly on
\(\varepsilon\), so does \(T_\varepsilon\). Moreover,
\[
T_0=L_0|_E=168\,I_E.
\]

It remains to compute the derivative \(T_0'\). Let
\[
K=W_0':E\to H_0.
\]
Differentiating
\[
T_\varepsilon=W_\varepsilon^{-1}\widetilde L_\varepsilon W_\varepsilon|_E
\]
at \(\varepsilon=0\), we obtain
\[
T_0'=-K\,L_0|_E+\widetilde L_0'|_E+L_0K|_E.
\]
Let \(u,v\in E\). Since \(L_0v=168\,v\), we have
\[
\langle -K L_0u+L_0Ku,v\rangle_{H_0}
=
-168\langle Ku,v\rangle_{H_0}+\langle Ku,L_0v\rangle_{H_0}
=
0.
\]
Therefore
\[
\langle T_0' u,v\rangle_{H_0}
=
\langle \widetilde L_0' u,v\rangle_{H_0}
\qquad\forall u,v\in E.
\]

We now compute \(\widetilde L_0'\) on \(E\). For \(u,v\in E\), set
\[
u_\varepsilon=U_\varepsilon^{-1}u=e^{-\frac32\varepsilon\rho}u,
\qquad
v_\varepsilon=U_\varepsilon^{-1}v=e^{-\frac32\varepsilon\rho}v.
\]
By definition of \(\widetilde L_\varepsilon\),
\[
\langle \widetilde L_\varepsilon u,v\rangle_{H_0}
=
\langle L_\varepsilon u_\varepsilon,v_\varepsilon\rangle_{L^2(g_\varepsilon)}
=
a_\varepsilon(u_\varepsilon,v_\varepsilon).
\]
Differentiating at \(\varepsilon=0\), using
\[
u_0'=-\frac32\rho u,\qquad v_0'=-\frac32\rho v,
\]
gives
\[
\langle \widetilde L_0' u,v\rangle_{H_0}
=
a_0'(u,v)+a_0(u_0',v)+a_0(u,v_0').
\]
Hence
\[
\langle \widetilde L_0' u,v\rangle_{H_0}
=
a_0'(u,v)-\frac32 a_0(\rho u,v)-\frac32 a_0(u,\rho v).
\]

Since \(u,v\in E\), we have
\[
L_0u=168\,u,\qquad L_0v=168\,v.
\]
Using the identity
\[
a_0(f,g)=\langle f,L_0g\rangle_{H_0}
\qquad (f,g\in C^\infty(M)),
\]
we obtain
\[
a_0(\rho u,v)=\langle \rho u,L_0v\rangle_{H_0}
=168\int_M \rho\,u\,v\,dV_{g_0}.
\]
Also, by symmetry of the Dirichlet form \(a_0\),
\[
a_0(u,\rho v)=a_0(\rho v,u)=\langle \rho v,L_0u\rangle_{H_0}
=168\int_M \rho\,u\,v\,dV_{g_0}.
\]
Therefore
\[
\langle \widetilde L_0' u,v\rangle_{H_0}
=
a_0'(u,v)-3\cdot 168\int_M \rho\,u\,v\,dV_{g_0}.
\]
Since
\[
m_0'(u,v)=3\int_M \rho\,u\,v\,dV_{g_0},
\]
it follows that
\[
\langle \widetilde L_0' u,v\rangle_{H_0}
=
a_0'(u,v)-168\,m_0'(u,v).
\]
By the definition of \(B^{(\rho)}\),
\[
\langle \widetilde L_0' u,v\rangle_{H_0}
=
\langle B^{(\rho)}u,v\rangle_{H_0}.
\]
Hence
\[
\langle T_0' u,v\rangle_{H_0}
=
\langle B^{(\rho)}u,v\rangle_{H_0}
\qquad\forall u,v\in E,
\]
so
\[
T_0'=B^{(\rho)}.
\]

Since \(T_\varepsilon\) is smooth in \(\varepsilon\), we conclude that
\[
T_\varepsilon
=
T_0+\varepsilon T_0'+o(\varepsilon)
=
168\,I_E+\varepsilon B^{(\rho)}+o(\varepsilon)
\]
in operator norm on \(E\). This proves (4).
\end{proof}

\begin{theorem}[Branch selection from a simple lowest splitter]
\label{thm:partII-branch-selection}
Let \(\rho\in C^\infty(M)\), and consider
\[
g_\varepsilon=e^{2\varepsilon\rho}g_0.
\]
Let \(B^{(\rho)}:E\to E\) be the first-order splitting operator on
\(E=E_{168}(g_0)\). Assume that \(B^{(\rho)}\) has a simple lowest eigenvalue,
with normalized eigenvector \(v_*\in E\).

Then for all sufficiently small \(\varepsilon>0\),
\begin{enumerate}[label=\arabic*.]
\item the first positive eigenvalue \(\lambda_1(g_\varepsilon)\) is simple;

\item an \(L^2(M,g_0)\)-normalized eigenvector \(\widetilde\phi_\varepsilon\in E_\varepsilon\)
of \(\widetilde L_\varepsilon\) associated with the lowest eigenvalue in the
perturbed \(168\)-cluster converges, after sign choice, to \(v_*\in E\) in
\(H_0=L^2(M,g_0)\);

\item consequently, an \(L^2(M,g_\varepsilon)\)-normalized first eigenfunction
\(\phi_\varepsilon\) of \(-\Delta_{g_\varepsilon}\) converges, after sign
choice, to \(v_*\) in \(C^\infty(M)\).
\end{enumerate}
\end{theorem}

\begin{proof}
Let \(T_\varepsilon\) be the finite-dimensional transported operator of
\Cref{prop:partII-cluster-reduction}. By that proposition,
\[
T_\varepsilon=168\,I_E+\varepsilon B^{(\rho)}+o(\varepsilon)
\qquad\text{in operator norm on }E.
\]
Since \(B^{(\rho)}\) has a simple lowest eigenvalue, standard
finite-dimensional perturbation theory implies that, for all sufficiently small
\(\varepsilon>0\), the operator \(T_\varepsilon\) has a simple lowest
eigenvalue, and its normalized lowest eigenvector \(u_\varepsilon\in E\) may be
chosen so that
\[
u_\varepsilon\to v_*
\qquad\text{in }E.
\]

By \Cref{prop:partII-cluster-reduction}(1), for all sufficiently small
\(\varepsilon>0\), the spectrum of \(\widetilde L_\varepsilon\) in the interval
\((\gamma_-,\gamma_+)\) consists exactly of the perturbed \(168\)-cluster, and
there is no positive spectrum below \(\gamma_-\). Therefore the lowest
eigenvalue in that cluster is exactly the first positive eigenvalue of
\(\widetilde L_\varepsilon\), hence also of \(-\Delta_{g_\varepsilon}\). Since
that lowest eigenvalue of \(T_\varepsilon\) is simple, it follows that
\(\lambda_1(g_\varepsilon)\) is simple for all sufficiently small
\(\varepsilon>0\). This proves \((1)\).

Now define
\[
\widetilde\phi_\varepsilon=W_\varepsilon u_\varepsilon\in E_\varepsilon,
\]
where \(W_\varepsilon:E\to E_\varepsilon\) is the unitary spectral transport of
\Cref{prop:partII-cluster-reduction}. Then \(\widetilde\phi_\varepsilon\) is a
normalized eigenvector of \(\widetilde L_\varepsilon\) associated with the
lowest eigenvalue in the perturbed \(168\)-cluster. Since
\[
W_\varepsilon\to \Id_E
\qquad\text{in operator norm on }E,
\]
and \(u_\varepsilon\to v_*\) in \(E\), we obtain
\[
\widetilde\phi_\varepsilon\to v_*
\qquad\text{in }H_0=L^2(M,g_0),
\]
after the same sign choice. This proves \((2)\).

Finally, define
\[
\phi_\varepsilon=U_\varepsilon^{-1}\widetilde\phi_\varepsilon
\in L^2(M,g_\varepsilon),
\]
where
\[
U_\varepsilon:L^2(M,g_\varepsilon)\to H_0,
\qquad
U_\varepsilon f=e^{\frac32\varepsilon\rho}f,
\]
is the unitary identification. Then \(\phi_\varepsilon\) is an
\(L^2(M,g_\varepsilon)\)-normalized first eigenfunction of
\(-\Delta_{g_\varepsilon}\). Since \(g_\varepsilon\to g_0\) smoothly and the
corresponding eigenvalues remain in the fixed compact interval
\((\gamma_-,\gamma_+)\), elliptic regularity yields uniform \(H^k\)-bounds for
\(\phi_\varepsilon\) for every \(k\). Passing to a subsequence if necessary and
using the \(L^2\)-convergence already proved, we conclude that
\[
\phi_\varepsilon\to v_*
\qquad\text{in }C^\infty(M)
\]
after sign choice. Since the limit is unique, the whole family converges. This
proves \((3)\).
\end{proof}
\begin{proposition}
\label{prop:partII-coefficient-functions}
For each \(j=0,\dots,12\), the coefficient function \(A_j\) is right
\(I^*\)-invariant, descends to a smooth function on
\[
M=S^3/I^*,
\]
belongs to the complexified first eigenspace \(E_\C\), and is a pure
left-weight vector of weight \(12-2j\).
\end{proposition}

\begin{proof}
This is exactly \Cref{prop:partI-coefficient-functions}.
\end{proof}
\begin{definition}[Explicit splitting directions]
\label{def:partII-coefficient-functions}
Define
\[
F_1=A_1,\qquad F_3=A_3,\qquad F_4=A_4.
\]
Then
\[
F_1,\ F_3,\ F_4\in E_\C,
\qquad
\Re(F_1),\ \Re(F_3),\ \Re(F_4)\in E.
\]
By \Cref{prop:partI-coefficient-functions}, their left weights are
\[
10,\qquad 6,\qquad 4,
\]
respectively.
\end{definition}

\begin{remark}
The repaired Part~I uses \(A_0\) to construct the order-\(2\) splitting
direction there. In Part~II the nearby minimal-Morse construction uses instead
the three explicit directions \(A_1,A_3,A_4\), because their weights are adapted
to the three exceptional fibers.
\end{remark}

\section{Restriction-selection on the exceptional fibers}
\label{sec:partII-restriction-selection}

We now use the exceptional-fiber period statement already proved in Part~I.

\begin{lemma}[Restriction-selection by stabilizer divisibility]
\label{lem:partII-weight-divisibility}
Let \(C_m\subset M\) be an exceptional Hopf fiber of order \(m\in\{5,3,2\}\).
If \(G\in E_\C\) is a pure left-weight vector of weight \(\ell\), then the
restriction of \(G\) to \(C_m\) vanishes unless \(2m\mid \ell\). In
particular:
\[
F_1|_{C_3}=F_1|_{C_2}=0,
\qquad
F_3|_{C_5}=F_3|_{C_2}=0,
\qquad
F_4|_{C_5}=F_4|_{C_3}=0.
\]
\end{lemma}

\begin{proof}
By \Cref{lem:partI-exceptional-fiber-period}, a pure left-weight vector of
weight \(\ell\) restricts nontrivially to \(C_m\) only if \(2m\mid \ell\).
Now \(F_1,F_3,F_4\) have weights \(10,6,4\), respectively, by
\Cref{def:partII-coefficient-functions}. The stated vanishing relations follow
immediately.
\end{proof}

\begin{lemma}[Target nonvanishing on the three exceptional fibers]
\label{lem:partII-target-nonvanishing}
One has
\[
F_1|_{C_5}\neq 0,\qquad
F_3|_{C_3}\neq 0,\qquad
F_4|_{C_2}\neq 0.
\]
Moreover, each surviving restriction descends to a nonzero trigonometric mode
of quotient frequency \(1\) on the corresponding exceptional circle.
\end{lemma}

\begin{proof}
\medskip

\noindent\textbf{The order-\(5\) fiber.}
Take the vertex lift
\[
z_5=(1,0)\in S^3.
\]
Then
\[
I_{12,z_5}(x,y)=I_{12}(x,y),
\]
so
\[
F_1(z_5)=A_1(z_5)=1\neq0.
\]
Hence \(F_1|_{C_5}\neq0\).

\medskip

\noindent\textbf{The order-\(2\) fiber.}
Take the edge-midpoint lift
\[
z_2=\frac{1}{\sqrt2}(i,1).
\]
Then
\[
\alpha=\frac{i}{\sqrt2},
\qquad
\beta=\frac{1}{\sqrt2},
\qquad
\overline\beta=\frac{1}{\sqrt2},
\qquad
\overline\alpha=-\frac{i}{\sqrt2}.
\]
Therefore
\[
I_{12,z_2}(x,y)
=
I_{12}\!\left(\frac{ix-y}{\sqrt2},\,\frac{x-iy}{\sqrt2}\right).
\]
Now \(F_4(z_2)=A_4(z_2)\) is the coefficient of \(x^8y^4\) in this polynomial.
A direct expansion, recorded in
Appendix~\ref{appendix:explicit-computations}, gives
\[
A_4(z_2)= -\frac{165}{64}-\frac{165}{32}i.
\]
In particular,
\[
F_4(z_2)\neq0.
\]
Hence \(F_4|_{C_2}\neq0\).

\medskip

\noindent\textbf{The order-\(3\) fiber.}
Define
\[
r=\sqrt{\frac{1+1/\sqrt3}{2}},
\qquad
s=\sqrt{\frac{1-1/\sqrt3}{2}},
\qquad
z_3=(r,\ s e^{-i\pi/4})\in S^3.
\]
Under the Hopf map
\[
\mathfrak h(\alpha,\beta)
=
\bigl(2\Re(\alpha\overline\beta),\,2\Im(\alpha\overline\beta),\,|\alpha|^2-|\beta|^2\bigr),
\]
one computes
\[
\mathfrak h(z_3)=\left(\frac1{\sqrt3},\frac1{\sqrt3},\frac1{\sqrt3}\right),
\]
which is a face-center point. Therefore the Hopf fiber through \(z_3\) is an
order-\(3\) exceptional fiber.

Now
\[
\beta=s e^{-i\pi/4}=s\frac{1-i}{\sqrt2},
\qquad
\overline\beta=s\frac{1+i}{\sqrt2}.
\]
Hence
\[
I_{12,z_3}(x,y)
=
I_{12}\!\left(rx-s\frac{1+i}{\sqrt2}y,\ s\frac{1-i}{\sqrt2}x+ry\right).
\]
The coefficient of \(x^9y^3\) is \(F_3(z_3)=A_3(z_3)\), and a direct expansion,
recorded in Appendix~\ref{appendix:explicit-computations}, yields
\[
A_3(z_3)
=
-\frac{55\sqrt3}{216}-\frac{55}{216}
+\left(\frac{55\sqrt3}{216}-\frac{55}{216}\right)i.
\]
In particular,
\[
F_3(z_3)\neq0.
\]
Hence \(F_3|_{C_3}\neq0\).

\medskip

Finally, by \Cref{lem:partI-exceptional-fiber-period}, the quotient frequency of
a surviving restriction of a weight-\(\ell\) vector to an order-\(m\) fiber is
\[
\frac{\ell}{2m}.
\]
Therefore the surviving restrictions on
\[
C_5,\qquad C_3,\qquad C_2
\]
have quotient frequency
\[
\frac{10}{10}=1,\qquad \frac{6}{6}=1,\qquad \frac{4}{4}=1,
\]
respectively.
\end{proof}

\section{Explicit nearby minimal-Morse lines}
\label{sec:partII-explicit-minimal-line}

We now prove the explicit geometric input. The starting point is the global
critical-circle description of the invariant seed established in Part~I:
\Cref{lem:partI-critical-set-F0,lem:partI-F0-morse-bott-all}.

\begin{theorem}[Explicit nearby minimal-Morse line]
\label{thm:partII-explicit-nearby-minimal-line}
There exist arbitrarily small nonzero parameters
\[
(\varepsilon_2,\varepsilon_3,\varepsilon_5)\in\R^3
\]
such that the function
\[
\Psi_{\varepsilon_2,\varepsilon_3,\varepsilon_5}
=
F_0+\varepsilon_2\,\Re(F_4)+\varepsilon_3\,\Re(F_3)+\varepsilon_5\,\Re(F_1)
\]
is Morse with exactly six critical points on \(M=S^3/I^*\). Equivalently, its
projective line is a minimal-Morse line in \(E\), and such lines may be chosen
arbitrarily close to the seed line \([F_0]\).
\end{theorem}

\begin{proof}
By \Cref{lem:partI-critical-set-F0,lem:partI-F0-morse-bott-all}, the invariant
seed satisfies
\[
\Crit(F_0)=C_5\sqcup C_3\sqcup C_2,
\]
and \(F_0\) is Morse--Bott along each of these three critical circles, with
nondegenerate transverse Hessian.

Choose pairwise disjoint tubular neighborhoods
\[
U_5,\qquad U_3,\qquad U_2
\]
of \(C_5,C_3,C_2\), respectively, so small that the Morse--Bott normal form
holds in each one. Let
\[
K=M\setminus(U_5\cup U_3\cup U_2).
\]
Since \(K\) is compact and contains no critical point of \(F_0\), there exists
\(c>0\) such that
\[
|\nabla F_0|_{g_0}\ge c
\qquad\text{on }K.
\]
Because \(E\) is finite-dimensional, the perturbation
\[
\varepsilon_2\,\Re(F_4)+\varepsilon_3\,\Re(F_3)+\varepsilon_5\,\Re(F_1)
\]
tends to zero in \(C^1\) as
\[
(\varepsilon_2,\varepsilon_3,\varepsilon_5)\to(0,0,0).
\]
Hence, for all sufficiently small parameters,
\[
|\nabla \Psi_{\varepsilon_2,\varepsilon_3,\varepsilon_5}|_{g_0}\ge \frac c2>0
\qquad\text{on }K.
\]
Therefore \(\Psi_{\varepsilon_2,\varepsilon_3,\varepsilon_5}\) has no critical
points outside \(U_5\cup U_3\cup U_2\).

Now fix \(m\in\{5,3,2\}\). By \Cref{lem:partI-F0-morse-bott-all}, \(F_0\) is
Morse--Bott along \(C_m\). Therefore there exist local coordinates
\[
(\theta,x,y)\in S^1\times\R^2
\]
on \(U_m\), with \(C_m=\{x=y=0\}\), such that
\[
F_0(\theta,x,y)=c_m+Q_m(x,y)+R_m(\theta,x,y),
\qquad
R_m(\theta,x,y)=O(\|(x,y)\|^3),
\]
where \(Q_m\) is a nondegenerate quadratic form in the normal variables.

Restrict \(\Psi_{\varepsilon_2,\varepsilon_3,\varepsilon_5}\) to \(C_m\). By
\Cref{lem:partII-weight-divisibility}, all cross-terms vanish, and by
\Cref{lem:partII-target-nonvanishing}, exactly one perturbation term survives:
\[
\Psi_{\varepsilon_2,\varepsilon_3,\varepsilon_5}|_{C_5}
=
F_0|_{C_5}+\varepsilon_5\,\Re(F_1)|_{C_5},
\]
\[
\Psi_{\varepsilon_2,\varepsilon_3,\varepsilon_5}|_{C_3}
=
F_0|_{C_3}+\varepsilon_3\,\Re(F_3)|_{C_3},
\]
\[
\Psi_{\varepsilon_2,\varepsilon_3,\varepsilon_5}|_{C_2}
=
F_0|_{C_2}+\varepsilon_2\,\Re(F_4)|_{C_2}.
\]

Each surviving complex restriction is a nonzero quotient frequency-\(1\) mode on
the corresponding quotient circle. We now show directly that the corresponding
real parts are nontrivial.

At the order-\(5\) fiber, we have
\[
F_1(z_5)=A_1(z_5)=1,
\]
so
\[
\Re(F_1)|_{C_5}\not\equiv 0.
\]

At the order-\(2\) fiber, by
\Cref{lem:appendix-explicit-computations},
\[
A_4(z_2)=-\frac{165}{64}-\frac{165}{32}i,
\]
hence
\[
\Re(F_4)(z_2)=-\frac{165}{64}\neq 0,
\]
so
\[
\Re(F_4)|_{C_2}\not\equiv 0.
\]

At the order-\(3\) fiber, again by
\Cref{lem:appendix-explicit-computations},
\[
A_3(z_3)
=
-\frac{55\sqrt3}{216}-\frac{55}{216}
+
\left(\frac{55\sqrt3}{216}-\frac{55}{216}\right)i,
\]
so
\[
\Re(F_3)(z_3)\neq 0,
\]
and hence
\[
\Re(F_3)|_{C_3}\not\equiv 0.
\]

Therefore each of the real-valued circle restrictions
\[
\Re(F_1)|_{C_5},\qquad \Re(F_3)|_{C_3},\qquad \Re(F_4)|_{C_2}
\]
is a nonzero real trigonometric mode of quotient frequency \(1\), and thus each
has exactly two nondegenerate critical points on the corresponding quotient
circle.

By \Cref{lem:MB-circle-splitting}, for all sufficiently small nonzero
parameters, each of those two circle-critical points lifts uniquely to a nearby
nondegenerate critical point of the full function in the corresponding tubular
neighborhood, and there are no others there. Therefore
\[
\#\Crit(\Psi_{\varepsilon_2,\varepsilon_3,\varepsilon_5}\cap U_5)=2,
\qquad
\#\Crit(\Psi_{\varepsilon_2,\varepsilon_3,\varepsilon_5}\cap U_3)=2,
\qquad
\#\Crit(\Psi_{\varepsilon_2,\varepsilon_3,\varepsilon_5}\cap U_2)=2.
\]

Since there are no critical points outside \(U_5\cup U_3\cup U_2\), it follows
that
\[
\#\Crit(\Psi_{\varepsilon_2,\varepsilon_3,\varepsilon_5})=2+2+2=6.
\]
All six are nondegenerate, so
\[
\Psi_{\varepsilon_2,\varepsilon_3,\varepsilon_5}
\]
is Morse with exactly six critical points.

By \Cref{prop:partI-minimal-six}, a Morse function on \(M=S^3/I^*\) is minimal
if and only if it has exactly six critical points. Hence
\[
\Psi_{\varepsilon_2,\varepsilon_3,\varepsilon_5}
\]
is minimal Morse.

Since the parameters may be chosen arbitrarily small and nonzero, the
corresponding projective lines are arbitrarily close to \([F_0]\). This proves
the theorem.
\end{proof}

\begin{corollary}
\label{cor:partII-spectral-consequence-from-explicit-line}
There exist minimal-Morse lines in \(E_{168}(g_0)\) arbitrarily close to
\([F_0]\). Consequently, by \Cref{cor:partII-realizable-neighborhood}, these
lines are realized as simple lowest eigendirections of explicit first-order
conformal splitting operators in \(\mathscr B\).
\end{corollary}

\begin{proof}
The first assertion is exactly \Cref{thm:partII-explicit-nearby-minimal-line}.
The spectral realizability statement then follows from
\Cref{cor:partII-realizable-neighborhood}.
\end{proof}

\section{Perturbative existence beyond the spherical metric}
\label{sec:partII-perturbative-existence}

We can now conclude the full perturbative theorem.

\begin{theorem}[Perturbative existence beyond the spherical metric]
\label{thm:partII-perturbative-existence}
There exists a smooth Riemannian metric \(g\) on
\[
M=S^3/I^*
\]
arbitrarily \(C^\infty\)-close to the spherical metric \(g_0\) such that:
\begin{enumerate}[label=\arabic*.]
\item the first positive eigenvalue \(\lambda_1(g)\) is simple;
\item an \(L^2(M,g)\)-normalized first eigenfunction is Morse with exactly six
critical points;
\item hence that first eigenfunction is minimal Morse;
\item consequently \((M,g)\) has property \(P\).
\end{enumerate}
\end{theorem}

\begin{proof}
By \Cref{cor:partII-spectral-consequence-from-explicit-line}, there exist
minimal-Morse lines in \(E\) arbitrarily close to \([F_0]\). Choose one such
line and write it as
\[
[\Psi]\subset E,
\]
where \(\Psi\in E\) is minimal Morse.

By \Cref{cor:partII-realizable-neighborhood}, every projective line
sufficiently close to \([F_0]\) is realized as the simple lowest eigendirection
of some operator in \(\mathscr B\). Hence there exists an operator
\[
A\in\mathscr B
\]
whose lowest eigenspace is exactly the line \([\Psi]\), and whose lowest
eigenvalue is simple.

By \Cref{prop:partII-every-B-realizable-by-rho}, there exists
\[
\rho\in C^\infty(M)
\]
such that
\[
B^{(\rho)}=A.
\]

Let
\[
g_\varepsilon=e^{2\varepsilon\rho}g_0.
\]
Then by \Cref{thm:partII-branch-selection}, for all sufficiently small
\(\varepsilon>0\), the first positive eigenvalue \(\lambda_1(g_\varepsilon)\)
is simple, and the corresponding normalized first eigenfunction
\(\phi_\varepsilon\) satisfies
\[
\phi_\varepsilon \to \Psi
\qquad\text{in }C^\infty(M).
\]

Since \(\Psi\) is Morse with exactly six critical points, standard
\(C^2\)-stability of nondegenerate critical points implies that, for all
sufficiently small \(\varepsilon>0\), the function \(\phi_\varepsilon\) is
again Morse with exactly six critical points. By
\Cref{cor:partII-six-critical-points}, it follows that \(\phi_\varepsilon\) is
minimal Morse. This proves \((1)\)--\((3)\).

Finally, since \(\lambda_1(g_\varepsilon)\) is simple and the corresponding
first eigenfunction is minimal Morse, \Cref{thm:partII-simple-first-eigenfunction-implies-P}
shows that \((M,g_\varepsilon)\) has property \(P\). This proves \((4)\).
\end{proof}

\appendix

\section{Differential of the lowest-eigenline map}
\label{appendix:eigenline-differential}

In this appendix we prove \Cref{prop:partII-differential-eigenline-map}.

\begin{proof}[Proof of \Cref{prop:partII-differential-eigenline-map}]
Let \(A\in\Omega\), let \(v\in E\) be a unit lowest eigenvector, and write
\[
\lambda=\lambda_1(A).
\]
Since \(\lambda\) is simple, we have the orthogonal decomposition
\[
E=\R v\oplus v^\perp,
\]
and the operator
\[
(A-\lambda I)\big|_{v^\perp}:v^\perp\to v^\perp
\]
is invertible.

We identify a neighborhood of \([v]\in \mathbb P(E)\) with a neighborhood of
\(0\in v^\perp\) via the chart
\[
w\in v^\perp \longmapsto [v+w].
\]
In this chart, the differential of the eigenline map is the derivative of the
\(v^\perp\)-component of the perturbed eigenvector.

Let \(H\in\mathscr B\) be arbitrary, and consider the perturbation
\[
A_t=A+tH.
\]
Because the lowest eigenvalue is simple, standard finite-dimensional
perturbation theory yields smooth functions
\[
\lambda(t),\qquad w(t)\in v^\perp,
\]
defined for \(|t|\) small, such that
\[
\lambda(0)=\lambda,\qquad w(0)=0,
\]
and
\[
A_t\bigl(v+w(t)\bigr)=\lambda(t)\bigl(v+w(t)\bigr).
\]
Thus, in the projective chart above,
\[
\ell(A_t)=[v+w(t)].
\]
Therefore
\[
d\ell_A(H)=w'(0)\in v^\perp.
\]

We now compute \(w'(0)\). Differentiate
\[
(A+tH)(v+w(t))=\lambda(t)(v+w(t))
\]
at \(t=0\). Using \(Av=\lambda v\) and \(w(0)=0\), we get
\[
Hv+Aw'(0)=\lambda'(0)v+\lambda w'(0).
\]
Rearranging,
\[
(A-\lambda I)w'(0)+Hv-\lambda'(0)v=0.
\]
Projecting orthogonally onto \(v^\perp\), we obtain
\[
(A-\lambda I)w'(0)+P_{v^\perp}(Hv)=0.
\]
Because \((A-\lambda I)|_{v^\perp}\) is invertible,
\[
w'(0)=-(A-\lambda I)^{-1}P_{v^\perp}(Hv).
\]
Hence
\[
d\ell_A(H)=-(A-\lambda_1(A)I)^{-1}P_{v^\perp}(Hv),
\]
which is the asserted formula.

For the surjectivity criterion, note that
\[
(A-\lambda_1(A)I)^{-1}:v^\perp\to v^\perp
\]
is an isomorphism. Therefore the image of \(d\ell_A\) is exactly
\[
(A-\lambda_1(A)I)^{-1}\bigl(P_{v^\perp}(\mathscr Bv)\bigr),
\]
so \(d\ell_A\) is surjective if and only if
\[
P_{v^\perp}(\mathscr Bv)=v^\perp.
\]
This proves the proposition.
\end{proof}

\section{Explicit coefficient computations}
\label{appendix:explicit-computations}

In this appendix we record the direct coefficient evaluations used in
\Cref{lem:partI-H2-restriction,lem:partI-H2-C3-C5,lem:partII-target-nonvanishing}.

\begin{lemma}
\label{lem:appendix-explicit-computations}
With
\[
I_{12}(x,y)=x^{11}y+11x^6y^6-xy^{11},
\]
the following identities hold.

\begin{enumerate}[label=\arabic*.]
\item For
\[
z_2=\frac1{\sqrt2}(i,1),
\]
one has
\[
I_{12}(z_2)=I_{12}\!\left(\frac{i}{\sqrt2},\frac1{\sqrt2}\right)
=
-\frac{11}{64}-\frac{i}{32}.
\]

\item For
\[
z_2=\frac1{\sqrt2}(i,1),
\]
the coefficient \(A_4(z_2)\) of \(x^8y^4\) in
\[
I_{12}\!\left(\frac{ix-y}{\sqrt2},\,\frac{x-iy}{\sqrt2}\right)
\]
is
\[
A_4(z_2)=-\frac{165}{64}-\frac{165}{32}i.
\]
In particular,
\[
\Re(A_4(z_2))=-\frac{165}{64}\neq 0.
\]

\item For
\[
r=\sqrt{\frac{1+1/\sqrt3}{2}},
\qquad
s=\sqrt{\frac{1-1/\sqrt3}{2}},
\qquad
z_3=(r,se^{-i\pi/4}),
\]
the coefficient \(A_3(z_3)\) of \(x^9y^3\) in
\[
I_{12}\!\left(rx-s\frac{1+i}{\sqrt2}y,\,
s\frac{1-i}{\sqrt2}x+ry\right)
\]
is
\[
A_3(z_3)
=
-\frac{55\sqrt3}{216}-\frac{55}{216}
+
\left(\frac{55\sqrt3}{216}-\frac{55}{216}\right)i.
\]
In particular,
\[
\Re(A_3(z_3))
=
-\frac{55(\sqrt3+1)}{216}\neq 0.
\]

\item For
\[
r=\sqrt{\frac{1+1/\sqrt3}{2}},
\qquad
s=\sqrt{\frac{1-1/\sqrt3}{2}},
\qquad
z_3=(r,se^{-i\pi/4}),
\]
one has
\[
A_0(z_3)=I_{12}(z_3)
=
I_{12}\!\left(r,\,s e^{-i\pi/4}\right)
=
\frac{11\sqrt3}{216}-\frac{i}{27}.
\]
In particular,
\[
A_0(z_3)\neq 0.
\]
\end{enumerate}
\end{lemma}

\begin{proof}
For (1), substitute
\[
x=\frac{i}{\sqrt2},
\qquad
y=\frac1{\sqrt2}
\]
into
\[
I_{12}(x,y)=x^{11}y+11x^6y^6-xy^{11}.
\]
Since
\[
x^{11}y=\frac{i^{11}}{2^6}=-\frac{i}{64},
\qquad
11x^6y^6=11\frac{i^6}{2^6}=-\frac{11}{64},
\qquad
-xy^{11}=-\frac{i}{64},
\]
we obtain
\[
I_{12}\!\left(\frac{i}{\sqrt2},\frac1{\sqrt2}\right)
=
-\frac{11}{64}-\frac{i}{32}.
\]

For (2), write
\[
u=\frac{ix-y}{\sqrt2},
\qquad
v=\frac{x-iy}{\sqrt2}.
\]
Then
\[
I_{12}(u,v)=u^{11}v+11u^6v^6-uv^{11}.
\]
We compute the coefficient of \(x^8y^4\) term-by-term.

First,
\[
u^{11}v
=
\frac1{2^6}(ix-y)^{11}(x-iy).
\]
To obtain \(x^8y^4\), we need either \(x^7y^4\) from \((ix-y)^{11}\) times \(x\),
or \(x^8y^3\) from \((ix-y)^{11}\) times \(-iy\). Thus the contribution is
\[
\frac1{64}\left[
\binom{11}{4} i^7(-1)^4
+
\binom{11}{3} i^8(-1)^3(-i)
\right]
=
\frac1{64}\left[-330i-165i\right]
=
-\frac{495}{64}i.
\]

Second,
\[
11u^6v^6=\frac{11}{2^6}(ix-y)^6(x-iy)^6.
\]
To obtain \(x^8y^4\), sum over all decompositions with \(a+b=4\):
\[
\frac{11}{64}\sum_{a=0}^4
\binom{6}{a}\binom{6}{4-a}
i^{6-a}(-1)^a(-i)^{4-a}
=
-\frac{165}{64}.
\]

Third,
\[
-uv^{11}=-\frac1{2^6}(ix-y)(x-iy)^{11}.
\]
To obtain \(x^8y^4\), we need either \(x^7y^4\) from \((x-iy)^{11}\) times \(ix\),
or \(x^8y^3\) from \((x-iy)^{11}\) times \(-y\). This gives
\[
-\frac1{64}\left[
i\binom{11}{4}(-i)^4
-
\binom{11}{3}(-i)^3
\right]
=
-\frac1{64}\left[330i+165i\right]
=
-\frac{495}{64}i.
\]

Adding the three contributions yields
\[
A_4(z_2)
=
-\frac{165}{64}
-\frac{495}{64}i
-\frac{495}{64}i
=
-\frac{165}{64}-\frac{990}{64}i
=
-\frac{165}{64}-\frac{165}{32}i.
\]
The real-part statement is immediate.

For (3), write
\[
u=rx-s\frac{1+i}{\sqrt2}y,
\qquad
v=s\frac{1-i}{\sqrt2}x+ry.
\]
Then
\[
I_{12}(u,v)=u^{11}v+11u^6v^6-uv^{11}.
\]
We compute the coefficient of \(x^9y^3\) term-by-term.

Set
\[
\alpha=s\frac{1+i}{\sqrt2},
\qquad
\beta=s\frac{1-i}{\sqrt2}.
\]
Then
\[
u=rx-\alpha y,
\qquad
v=\beta x+ry.
\]

\medskip

\noindent\textbf{Contribution from \(u^{11}v\).}
To obtain \(x^9y^3\), there are two possibilities:

\begin{enumerate}[label=\roman*.]
\item choose \(x^8y^3\) from \(u^{11}\) and \(\beta x\) from \(v\);
\item choose \(x^9y^2\) from \(u^{11}\) and \(ry\) from \(v\).
\end{enumerate}

Thus the contribution is
\[
\binom{11}{3}r^8(-\alpha)^3\beta+\binom{11}{2}r^{10}(-\alpha)^2.
\]
Using
\[
(1+i)^2=2i,\qquad (1+i)^3=-2+2i,
\]
one finds
\[
(-\alpha)^2=s^2 i,
\qquad
(-\alpha)^3\beta=-s^4 i.
\]
Hence
\[
C_1
=
-165\,r^8s^4\,i+55\,r^{10}s^2\,i.
\]

\medskip

\noindent\textbf{Contribution from \(11u^6v^6\).}
We must choose terms whose total \(y\)-degree is \(3\). Let \(a\) be the number
of \(y\)'s taken from \(u^6\), so that \(3-a\) are taken from \(v^6\). Then
\[
C_2
=
11\sum_{a=0}^3
\binom{6}{a}\binom{6}{3-a}
r^{9-2a}\,(-\alpha)^a\,\beta^{3+a}.
\]
Since
\[
(-\alpha)^a\beta^{3+a}
=
(-1)^a s^{3+2a}\frac{(1+i)^a(1-i)^{3+a}}{2^{(2a+3)/2}},
\]
and
\[
(1+i)^a(1-i)^a=2^a,
\qquad
(1-i)^3=-2-2i,
\]
this becomes
\[
C_2
=
11(-2-2i)2^{-3/2}
\sum_{a=0}^3
(-1)^a
\binom{6}{a}\binom{6}{3-a}
r^{9-2a}s^{3+2a}.
\]
The inner sum is
\[
20r^9s^3-90r^7s^5+90r^5s^7-20r^3s^9
=
10r^3s^3(r^2-s^2)^3.
\]
Since
\[
r^2-s^2=\frac1{\sqrt3},
\]
we obtain
\[
C_2
=
11(-2-2i)2^{-3/2}\cdot \frac{10r^3s^3}{3\sqrt3}.
\]

\medskip

\noindent\textbf{Contribution from \(-uv^{11}\).}
Again there are two possibilities:

\begin{enumerate}[label=\roman*.]
\item choose \(-\alpha y\) from \(u\) and \(x^9y^2\) from \(v^{11}\);
\item choose \(rx\) from \(u\) and \(x^8y^3\) from \(v^{11}\).
\end{enumerate}

Thus the contribution is
\[
-\Bigl(\binom{11}{2}(-\alpha)\beta^9r^2+\binom{11}{3}r\beta^8r^3\Bigr).
\]
Using
\[
(1-i)^2=-2i,\qquad (1-i)^3=-2-2i,
\]
one finds after simplification
\[
C_3=-165\,r^4s^8\,i+55\,r^2s^{10}\,i.
\]

\medskip

Therefore
\[
A_3(z_3)=C_1+C_2+C_3.
\]
Now use
\[
r^2=\frac{1+1/\sqrt3}{2},
\qquad
s^2=\frac{1-1/\sqrt3}{2},
\qquad
r^2+s^2=1,
\qquad
r^2s^2=\frac16,
\qquad
r^2-s^2=\frac1{\sqrt3}.
\]
A direct simplification gives
\[
A_3(z_3)
=
-\frac{55\sqrt3}{216}-\frac{55}{216}
+
\left(\frac{55\sqrt3}{216}-\frac{55}{216}\right)i.
\]
The real-part statement follows immediately.

For (4), we evaluate
\[
I_{12}(x,y)=x^{11}y+11x^6y^6-xy^{11}
\]
at
\[
x=r,
\qquad
y=s e^{-i\pi/4},
\qquad
r=\sqrt{\frac{1+1/\sqrt3}{2}},
\qquad
s=\sqrt{\frac{1-1/\sqrt3}{2}}.
\]
Since
\[
e^{-i\pi/4}=\frac{1-i}{\sqrt2},
\qquad
e^{-i6\pi/4}=e^{-3\pi i/2}=i,
\qquad
e^{-i11\pi/4}=e^{-i3\pi/4}=\frac{-1-i}{\sqrt2},
\]
we get
\[
I_{12}(z_3)
=
\frac{r^{11}s}{\sqrt2}(1-i)
+11i\,r^6s^6
+\frac{rs^{11}}{\sqrt2}(1+i).
\]
Therefore
\[
\Re I_{12}(z_3)=\frac{rs}{\sqrt2}(r^{10}+s^{10}),
\qquad
\Im I_{12}(z_3)=\frac{rs}{\sqrt2}(-r^{10}+s^{10})+11r^6s^6.
\]

Now set
\[
a=r^2=\frac{1+1/\sqrt3}{2},
\qquad
b=s^2=\frac{1-1/\sqrt3}{2}.
\]
Then
\[
a+b=1,\qquad ab=\frac16,\qquad a-b=\frac1{\sqrt3}.
\]
Also,
\[
\frac{rs}{\sqrt2}=\frac{\sqrt3}{6},
\qquad
r^6s^6=(ab)^3=\frac1{216}.
\]
A direct exact computation gives
\[
a^5+b^5=\frac{11}{36},
\qquad
b^5-a^5=-\frac{19\sqrt3}{108}.
\]
Hence
\[
r^{10}+s^{10}=\frac{11}{36},
\qquad
s^{10}-r^{10}=-\frac{19\sqrt3}{108}.
\]
Substituting into the previous formulas,
\[
\Re I_{12}(z_3)
=
\frac{\sqrt3}{6}\cdot \frac{11}{36}
=
\frac{11\sqrt3}{216},
\]
and
\[
\Im I_{12}(z_3)
=
\frac{\sqrt3}{6}\left(-\frac{19\sqrt3}{108}\right)+\frac{11}{216}
=
-\frac{19}{216}+\frac{11}{216}
=
-\frac1{27}.
\]
Therefore
\[
A_0(z_3)=I_{12}(z_3)=\frac{11\sqrt3}{216}-\frac{i}{27}\neq0.
\]
\end{proof}

\section{A Schur-complement reduction formula}
\label{appendix:schur-complement-reduction}

This appendix records the elementary block-Hessian identity used in the proof of
\Cref{lem:MB-circle-splitting}.

\begin{lemma}
\label{lem:appendix-schur-complement}
Let \(F(\theta,u)\) be a smooth function near \((\theta_0,u_0)\in \R\times\R^m\).
Assume that
\[
\nabla_u F(\theta_0,u_0)=0
\]
and that
\[
D_u^2F(\theta_0,u_0)
\]
is invertible. Then, after solving the normal critical-point equation
\[
\nabla_u F(\theta,\xi(\theta))=0
\]
for a smooth local graph \(u=\xi(\theta)\), the reduced function
\[
G(\theta)=F(\theta,\xi(\theta))
\]
satisfies
\[
G''(\theta_0)
=
\partial_{\theta\theta}F(\theta_0,u_0)
-
\partial_{\theta u}F(\theta_0,u_0)
\bigl(D_u^2F(\theta_0,u_0)\bigr)^{-1}
\partial_{u\theta}F(\theta_0,u_0).
\]
Consequently, the full Hessian of \(F\) at \((\theta_0,u_0)\) is nondegenerate
if and only if both
\[
D_u^2F(\theta_0,u_0)
\]
and
\[
G''(\theta_0)
\]
are nonzero/invertible.
\end{lemma}

\begin{proof}
Differentiate
\[
\nabla_u F(\theta,\xi(\theta))=0
\]
with respect to \(\theta\). This gives
\[
\partial_{u\theta}F(\theta,\xi(\theta))
+
D_u^2F(\theta,\xi(\theta))\,\xi'(\theta)=0.
\]
At \(\theta=\theta_0\),
\[
\xi'(\theta_0)
=
-\bigl(D_u^2F(\theta_0,u_0)\bigr)^{-1}
\partial_{u\theta}F(\theta_0,u_0).
\]

Now
\[
G'(\theta)
=
\partial_\theta F(\theta,\xi(\theta))
+
\langle \nabla_u F(\theta,\xi(\theta)),\xi'(\theta)\rangle
=
\partial_\theta F(\theta,\xi(\theta)),
\]
because \(\nabla_u F(\theta,\xi(\theta))=0\). Differentiating once more,
\[
G''(\theta)
=
\partial_{\theta\theta}F(\theta,\xi(\theta))
+
\partial_{\theta u}F(\theta,\xi(\theta))\,\xi'(\theta).
\]
Evaluating at \(\theta_0\) and substituting the formula for \(\xi'(\theta_0)\)
yields
\[
G''(\theta_0)
=
\partial_{\theta\theta}F(\theta_0,u_0)
-
\partial_{\theta u}F(\theta_0,u_0)
\bigl(D_u^2F(\theta_0,u_0)\bigr)^{-1}
\partial_{u\theta}F(\theta_0,u_0).
\]

The final statement is the standard Schur-complement criterion for the block
matrix of the Hessian of \(F\) at \((\theta_0,u_0)\).
\end{proof}

\end{document}